\documentclass[12pt]{article}
\usepackage{mathptmx}
\usepackage{amssymb,latexsym,amsmath,amsthm}

\usepackage{graphicx}
\usepackage{epsfig}
\usepackage{tikz}
\usetikzlibrary{shapes,backgrounds,calc}
\usepackage{cite}
\usepackage{enumerate}
\usepackage[pdftex,hypertexnames=false,linktocpage=true]{hyperref}
\hypersetup{colorlinks=true,linkcolor=black,anchorcolor=blue,citecolor=black,filecolor=blue,urlcolor=blue,bookmarksnumbered=true,pdfview=FitB}

\usepackage[top=1in, bottom=1in, left=1in, right=1in]{geometry}
\usepackage{boxedminipage}
\usepackage{verbatim}
\usepackage{enumitem}
\usepackage{fancyhdr}
\usepackage[utf8]{inputenc}
\usepackage{color}
\usepackage{mathtools}
\definecolor{myGreen}{rgb}{0,.5,0}
\definecolor{black}{rgb}{0,0,0}
\definecolor{myYellow}{rgb}{.6,.6,.1}
\definecolor{Cyan}{rgb}{.2,.6,.7}
\definecolor{white}{rgb}{1,1,1}
\definecolor{Purple}{rgb}{.4,0,1}
\definecolor{deepblack}{rgb}{.53,.29,.24}
\definecolor{Black}{rgb}{0,0,0}
\definecolor{Grey}{rgb}{.45,.45,.45}

\newcommand{\op}[1]{\left(#1\right)}

\newcommand{\abs}[1]{\left\vert#1\right\vert}

\newcommand{\C}{\mathcal{C}}

\newcommand{\multiset}[2]{\ensuremath{\left(\kern-.3em\left(\genfrac{}{}{0pt}{}{#1}{#2}\right)\kern-.3em\right)}}
\definecolor{cb1}{RGB}{230,159,0}
\definecolor{cb4}{RGB}{86,180,233}
\definecolor{cb2}{RGB}{0,158,115}
\definecolor{cb5}{RGB}{0,114,178}
\definecolor{cb3}{RGB}{213,94,0}
\definecolor{cb6}{RGB}{204,121,167}
\tikzset{vtx/.style={circle,black,fill=black,inner sep=1pt,minimum size=7pt}}
\tikzset{whiteVtx/.style={circle,white,fill=white,inner sep=1pt}}
\tikzset{goldVtx/.style={circle,white,fill=CUGold,inner sep=1pt}}
\tikzset{cb1Vtx/.style={circle,white,fill=cb1,inner sep=1pt}}
\tikzset{cb2Vtx/.style={circle,white,fill=cb2, inner sep=1pt}}
\tikzset{cb3Vtx/.style={circle,white,fill=cb3, inner sep=1pt}}
\tikzset{cb4Vtx/.style={circle,white,fill=cb4, inner sep=1pt}}
\tikzset{cb5Vtx/.style={circle,white,fill=cb5, inner sep=1pt}}
\tikzset{blackVtx/.style={circle,white,fill=black,inner sep=1pt}}
\tikzset{type/.style={draw=cb6,very thick,fill=yellow}}
\tikzset{type2/.style={draw=cb6,very thick,fill=yellow,inner sep=1pt,minimum size=8pt}}
\tikzset{type3/.style={circle,draw=black,fill=white,inner sep=1pt,minimum size=7pt}}
\tikzset{e1/.style={line width=4pt,cb3}}
\tikzset{e2/.style={line width=4pt,cb4}}
\tikzset{e3/.style={line width=4pt,cb2}}
\tikzset{eq/.style={line width=4pt,gray!40}}
\tikzset{dedge/.style={->,line width=1.3pt}} 
\newtheorem{theorem}{Theorem}
\newtheorem{thm}{}[section]

\newtheorem{claim}[thm]{Claim}
\newtheorem{conjecture}[theorem]{Conjecture}

\newtheorem{proposition}[theorem]{Proposition}
\newtheorem{problem}{Problem}

\numberwithin{subcase}{case}
\theoremstyle{definition}

\graphicspath{{./graphics/}}

\renewcommand{\C}{\mathcal{C}}

\newcommand{\OO}{\begin{tikzpicture}[scale=0.15]
	\fill[black]  (-1,0) arc (180:0:1) ;
	\fill[black]  (-1,0) arc (0:180:-1) ;
	\draw[black] (0,0) circle (1) ;
    \end{tikzpicture}}
\newcommand{\OI}{\begin{tikzpicture}[scale=0.15]
	\fill[black]  (-1,0) arc (180:0:1) ;
	\draw[black] (0,0) circle (1) ;
    \end{tikzpicture}}
\newcommand{\II}{\begin{tikzpicture}[scale=0.15]
	\draw[black] (0,0) circle (1) ;
	\draw[black] (-1,0) -- (1,0) ;
    \end{tikzpicture}}
\newcommand{\IO}{\begin{tikzpicture}[scale=0.15]
	\fill[black]  (-1,0) arc (0:180:-1) ;
	\draw[black] (0,0) circle (1) ;
    \end{tikzpicture}}

\makeatletter
\tikzset{circle split part fill/.style args={#1,#2}{%
 alias=tmp@name, 
  postaction={%
    insert path={
     \pgfextra{%
     \pgfpointdiff{\pgfpointanchor{\pgf@node@name}{center}}%
                  {\pgfpointanchor{\pgf@node@name}{east}}%
     \pgfmathsetmacro\insiderad{\pgf@x}
      \fill[#1] (\pgf@node@name.base) ([xshift=-\pgflinewidth]\pgf@node@name.east) arc
                          (0:180:\insiderad-\pgflinewidth)--cycle;
      \fill[#2] (\pgf@node@name.base) ([xshift=\pgflinewidth]\pgf@node@name.west)  arc
                           (180:360:\insiderad-\pgflinewidth)--cycle;            
         }}}}}  
 \makeatother  

\newcommand{\oururl}{\url{http://lidicky.name/pub/tournaments}}

\begin{document}


\title{Inducibility of 4-vertex tournaments}
	
\author{
Dalton Burke\thanks{Department of Mathematical and Statistical Sciences, University of Colorado Denver, E-mail: {\tt Dalton.Burke@ucdenver.edu} .} 
\and
Bernard Lidick\'{y}\thanks{Department of Mathematics, Iowa State University, Ames, IA, E-mail: {\tt lidicky@iastate.edu}. Research of this author is supported in part by NSF grants DMS-1855653 and DMS-2152490 and by Scott Hanna fellowship.}
\and
Florian Pfender\thanks{Department of Mathematical and Statistical Sciences, University of Colorado Denver, E-mail: {\tt 
Florian.Pfender@ucdenver.edu}. Research is partially supported by NSF grants DMS-1855622 and DMS-2152498.} 
\and
Michael Phillips\thanks{Department of Mathematical and Statistical Sciences, University of Colorado Denver, E-mail: {\tt Michael.2.Phillips@ucdenver.edu} .} 
}

\maketitle
	
\setlength{\abovedisplayskip}{5pt}
\setlength{\belowdisplayskip}{5pt}

\begin{abstract}
We determine the inducibility of all tournaments with at most $4$ vertices together with the extremal constructions.
The $4$-vertex tournament containing an oriented $C_3$ and one source vertex has a particularly interesting extremal construction, first conjectured by Bo\.{z}yk, Grzesik and Kielak. It is an unbalanced blow-up of an edge, where the sink vertex is replaced by a quasi-random tournament and the source vertex is iteratively replaced by a copy of the construction itself.
\end{abstract}

\section{Introduction}
One of the central questions in extremal graph theory is to maximize the number of induced copies of a given graph $H$ in a larger host graph on a fixed number of vertices. 
Denoting the number of vertices by of a graph $G$ by $|G|$,
let $I(H,G)$ be the number of vertex subsets of $G$ which induce a graph isomorphic to $H$, and let
\[
I(H,n)=\max_{|G|=n}I(H,G).
\]
We normalize these definitions and write $i(H,G)=\frac{I(H,G)}{{|G|\choose |H|}}$
 and $i(H,n)=\frac{I(H,n)}{{n\choose |H|}}$.
 This implies that $0\le i(H,G)\le 1$, and we can think of $i(H,G)$ as a subgraph density. An easy averaging argument shows that $i(H,n)$ is monotone decreasing and thus converges for $n \to \infty$. Pippenger and Golumbic~\cite{PippengerGolumbic} define the {\em inducibility of $H$} as the limit of this quantity,
\[
i(H)=\lim_{n\to\infty} i(H,n).
\]
Determining inducibilities is notoriously difficult, and the answer is known only for very few explicit graphs $H$. A major breakthrough for the problem was the introduction of the flag algebra method by Razborov \cite{razborov} in 2007, and since then the inducibility of a good number of small graphs has been determined with the help of this method~\cite{MR3386015,MR3425964,nets}. While we are using this method as well in this paper, we will not thoroughly explain it here but rather direct the reader to earlier papers~\cite{BDL2020,MR4271644,MR3934375}. In a nutshell, the method uses semidefinite programming to solve an optimization problem on subgraph densities which can be set up in a very structured and easily computer assisted way, almost to the point where one may call it fully automated. We can add any number of linear constraints on the subgraph densities to the semidefinite program. Nevertheless, we do not even know $i(P_4)$, i.e. the inducibility of the path on four vertices, and we do not even have a conjecture for the answer. 

On the other end of the spectrum, Fox, Huang, and Lee~\cite{Fox}, and independently Yuster~\cite{Yuster}, have determined exact values for $i(H,n)$ and thus $i(H)$ for all $n$ and almost all large  enough graphs $H$ by studying random graphs. They show that the extremal construction is an iterated blow-up of the given graph, a fractal like structure. This iterated blow-up construction was already established by Pippinger and Golumbic as a general lower bound for inducibilities, and they asked which graphs meet this lower bound.
There are numerous other results on inducibility~\cite{MR1292981,MR4040054,MR3861783,MR3856711,MR3667662,MR3269907,fox2019inducibility,lidicky2021c5,liu2020stability}.

All of these questions can be studied for directed graphs as well, the definitions naturally transfer. 
Falgas-Ravry and Vaughan~\cite{MR2988862} studied inducibility of small outstars using flag algebras. 
Huang~\cite{Huang2014} extended the result to all outstars. 
This was further generalized to other stars by Hu, Ma, Norin, and Wu~\cite{hu2020inducibility}. Short paths with further restrictions were considered in~\cite{MR4109639} and orientations of a 4-cycle in~\cite{HLPV202X}.
In an REU (Research Experience for Undergraduates) in 2018, Burgher and Burke studied and conjectured extremal constructions for most oriented graphs (directed graphs without $2$-cycles) of up to $4$ vertices using the flag algebra method. In a similar and independent project around the same time, Bo\.{z}yk, Grzesik and Kielak~\cite{bozyk2020inducibility} established the same and more bounds and constructions for oriented graphs.

In this paper, we look closer at the tournaments in this list, i.e. oriented complete graphs. The number of non-isomorphic tournaments on $k$ vertices is slightly smaller than the number of graphs, and flag algebra computations tend to have similar power. The two projects mentioned in the previous paragraph both found inducibility bounds and closely matching lower bound constructions for all tournaments on up to $4$ vertices, where the results are easy or trivial for all but three of these $8$ small tournaments. These last three tournaments on $4$ vertices have very interesting constructions, and in this paper we prove that these constructions are indeed optimal for large $n$.

In a somewhat related question, Mubayi and Razborov~\cite{Dhruv} considered edge colored tournaments and showed that for every tournament $T$ on $k \geq 4$ vertices whose edges are colored by $\binom{k}{2}$ distinct colors, the iterated blow-up of $T$ achieves $i(T,n)$. This implies that $i(T) =  \frac{k!}{k^k-k}$ in this rainbow setting.

\section{Results}

We discuss tournaments on at most four vertices.
For the tournaments $T_1$ and $T_2$ on one and two vertices, respectively, any tournament $T$ has $i(T_k,T)=1$, and thus $i(T_k)=i(T_k,n)=1$. Similarly, for all transitive tournaments $TT_k$ on $k\ge 3$ vertices, the transitive tournament $TT_n$ on $n\ge k$ vertices is the unique tournament on $n$ vertices with $i(TT_k,T)=1$, and thus $i(TT_k)=i(TT_k,n)=1$. On the other hand, $i(TT_3,T)$ is minimized exactly if $T$ has all out-degrees in $\{\frac{n-2}{2},\frac{n-1}{2},\frac{n}{2}\}$. This easily follows from counting $TT_3$ by first choosing the source vertex, and then any two out-neighbors. As a consequence, one gets for the only other tournament on three vertices $C_3$:
\begin{proposition}[Folklore]\label{C3}
 The number of induced copies of $C_3$ is maximized if and only if every vertex of a tournament has out-degree in $\{\frac{n-2}{2},\frac{n-1}{2},\frac{n}{2}\}$. 
 \end{proposition}

This implies $i(C_3) = 1/4$ and leaves us with three $4$-vertex tournaments to consider, see Figure~\ref{4vtxTours}: the tournaments we get from $C_3$ by adding a source vertex ($C_3^+$), a sink vertex ($C_3^-$), and by adding a vertex of out-degree $1$ or $2$ (this choice results in isomorphic outcomes, the carousel $C_4$ defined in the next paragraph).

\begin{figure}[ht] \centering
\begin{tikzpicture}[scale=1.5]
\node[vtx] (a1) at (0,0) {};
\node[vtx] (a2) at (0,1) {};
\node[vtx] (a3) at (1,0) {};
\node[vtx] (a4) at (1,1) {};

\node[vtx] (b1) at (0+3,0) {};
\node[vtx] (b2) at (0+3,1) {};
\node[vtx] (b3) at (1+3,0) {};
\node[vtx] (b4) at (1+3,1) {};

\node[vtx] (c1) at (0+6,0) {};
\node[vtx] (c2) at (0+6,1) {};
\node[vtx] (c3) at (1+6,0) {};
\node[vtx] (c4) at (1+6,1) {};

\node[vtx] (d1) at (0+9,0) {};
\node[vtx] (d2) at (0+9,1) {};
\node[vtx] (d3) at (1+9,0) {};
\node[vtx] (d4) at (1+9,1) {};

\draw[dedge] (a1) -> (a2); \draw[dedge] (a1) -> (a3); \draw[dedge] (a1) -> (a4);
\draw[dedge] (a2) -> (a3); \draw[dedge] (a2) -> (a4); \draw[dedge] (a3) -> (a4);

\draw[dedge] (b1) -> (b2); \draw[dedge] (b1) -> (b3); \draw[dedge] (b1) -> (b4);
\draw[dedge] (b2) -> (b3); \draw[dedge] (b4) -> (b2); \draw[dedge] (b3) -> (b4);

\draw[dedge] (c2) -> (c1); 
\draw[dedge] (c3) -> (c1); \draw[dedge] (c4) -> (c1);
\draw[dedge] (c2) -> (c3); \draw[dedge] (c4) -> (c2); \draw[dedge] (c3) -> (c4);

\draw[dedge] (d1) -> (d2); \draw[dedge] (d1) -> (d3); \draw[dedge] (d4) -> (d1);
\draw[dedge] (d2) -> (d3); \draw[dedge] (d4) -> (d2); \draw[dedge] (d3) -> (d4);

\node at (0.5,-0.5) {$TT_4$};
\node at (3.5,-0.5) {$C_3^+$};
\node at (6.5,-0.5) {$C_3^-$};
\node at (9.5,-0.5) {$C_4$};
\end{tikzpicture}
\caption{The four 4-vertex tournaments.}\label{4vtxTours}
\end{figure}
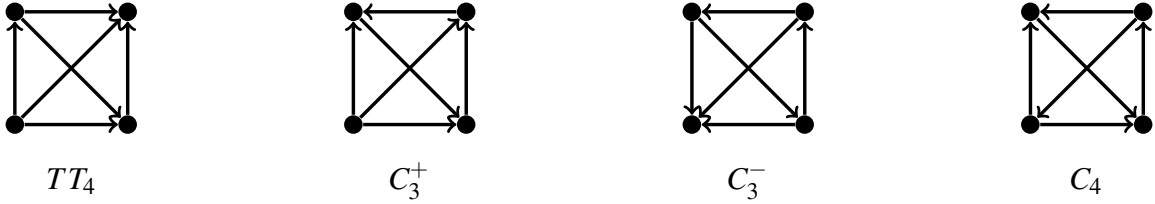

Let us now define the class $\C_n$ of {\em carousels} on $n \ge 3$ vertices. A tournament $T$ is in $\C_n$ if its vertices can be labeled $\{v_1,v_2,\ldots,v_n\}$ such that $v_iv_j\in E(T)$ if $0<j-i<\frac{n}{2}$ or if $-n<j-i<-\frac{n}{2}$. An easy exercise shows that a tournament $T$ is in $\C_n$ if and only if for every $x \in V(T)$, the in- and out-neighborhoods induce transitive tournaments ($T$ is \emph{locally transitive}) and are as balanced as possible ($T$ is \emph{balanced} when $|V(T)|$ is odd, or \emph{nearly balanced} when $|V(T)|$ is even). See Figure~\ref{fig:carousels} for an illustration.

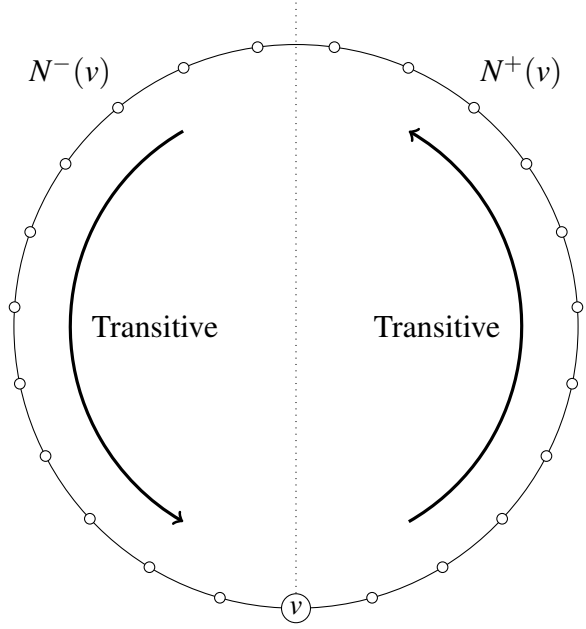
\begin{figure}
\begin{center}
\begin{tikzpicture}[scale=.75]
	\draw (0,0) circle (5cm);
	\node[inner sep=0.5mm,draw=black,fill=white,shape=circle] (v) at (0,-5) {\small $v$};
	\foreach \i in {1,...,22} {
		\node[inner sep=0.5mm,draw=black,fill=white,shape=circle] (v\i) at ({360/23 * \i - 90}:5) {};
	}
	\draw[dotted] (v) -- (0,5.8);
	\node at (-4,4.5) {$N^-(v)$}; \node at (4,4.5) {$N^+(v)$};
	\node at (2.5,0) {Transitive};
	\node at (-2.5,0) {Transitive};
	\draw[very thick,->] (2,{-sqrt(12)}) arc (-60:60:4cm);
	\draw[very thick,->] (-2,{sqrt(12)}) arc (120:240:4cm);
\end{tikzpicture}
\end{center}
\caption{For odd $n$, the carousel $C \in \C_n$ is unique and vertex transitive. For even $n$, the directions of the diagonals can be chosen arbitrarily, resulting in several non-isomorphic tournaments.}\label{fig:carousels}
\end{figure}

Observe that for odd $n$ and for $n=4$ (up to isomorphism), $\C_n$ contains exactly one tournament, and we will call this unique carousel $C_n$. For even $n\ge 6$, $\C_n$ contains more than one tournament, depending on the directions of the arcs $v_iv_{i+\frac{n}2}$. For even $n$, we denote by $C_n\in\C_n$ the unique tournament we get from deleting one vertex in $C_{n+1}$. Note that one can alternatively construct $C_n$ from $C_{n-1}$ by duplicating one vertex and adding the edge between the two otherwise identical vertices in either direction.

Our first result describes precisely all extremal constructions for $I(C_4,n)$ for large enough $n$.
\begin{theorem}\label{TC4_Full}
For $n\ge 4$, the tournaments maximizing $I(C_4,T)$ are precisely the tournaments in $\C_n$. Consequently,
$i(C_4)=\frac12$, and for every $n$, we have 
\[
I(C_4,n) = \begin{cases}
	\frac{n(n^2-1)(n-3)}{48} &  \text{if $n$ is odd,} \\
	\frac{n(n^2-4)(n-3)}{48} &  \text{if $n$ is even.}
\end{cases}
\]
\end{theorem}
Note that the asymptotic statement that $i(C_4)=\frac12$ is also proved in~\cite{MR3549509} and~\cite{bozyk2020inducibility}, with proofs very similar to the one we provide in the next section. Our contribution here is the proof of the exact construction. Numeric bounds from flag algebra computations indicate that a similar statement may also be true for $C_5$, $C_6$, $C_7$ and $C_8$, and we conjecture it is true for all $k$. 
See the discussion at the end of this paper for a few more details on this.
Observe that for $k\ge 5$ and even $n\ge k$, $C_n$ contains more copies of $C_k$ than the other members of $\C_n$, so our conjectured extremal tournament is unique for $k\ge 5$.
\begin{conjecture}\label{conTC}
For all $k\ge 5$ and $n\ge k$, the unique $n$-vertex tournaments maximizing $I(C_k,T)$ are the tournaments $C_n$. 
\end{conjecture}

The only tournaments on $4$ vertices left to consider are the two tournaments $C_3^-$ and $C_3^+$. As one gets $C_3^-$ from $C_3^+$ by reversal of all arcs, the tournaments extremal for $C_3^-$ are precisely the reversals of the tournaments extremal for $C_3^+$, so it suffices to only study $C_3^+$.
Consider the following probabilistic construction of a tournament $\tilde{T}_n$ on $n$ vertices which was discovered independently by Burgher and Burke, and in~\cite{bozyk2020inducibility} with an almost matching upper bound via the flag algebra method. For some fixed $\alpha\in (0,1)$, partition the vertices into two sets $H_n$ (for high out-degree) and $L_n$ (for low out-degree) of size $\lceil\alpha n\rceil$ and $\lfloor (1-\alpha)n\rfloor$, respectively. On the set $L_n$, direct the edges uniformly at random, i.e. insert a random tournament $R$ on ${\lfloor(1-\alpha)n\rfloor}$ vertices. All arcs between the sets are directed from $H_n$ to $L_n$. On the set $H_n$, iterate the construction, i.e. insert the tournament $\tilde{T}_{\lceil \alpha n\rceil}$ inductively. See Figure~\ref{fig:iter} for a sketch of the iterated construction.

Note that with probability approaching $1$ for large $n$, we have $i(H,R)=\mathbb{E}(i(H,R))+o(1)$ for every tournament $H$ in a random tournament $R$ on $n$ vertices. We may thus choose a (quasi-random) sequence of tournaments $R_n$ on $n$ vertices with $i(H,R_n)=\mathbb{E}(i(H,R))+o(1)$, and use this sequence in place of the probabilistic construction described above.

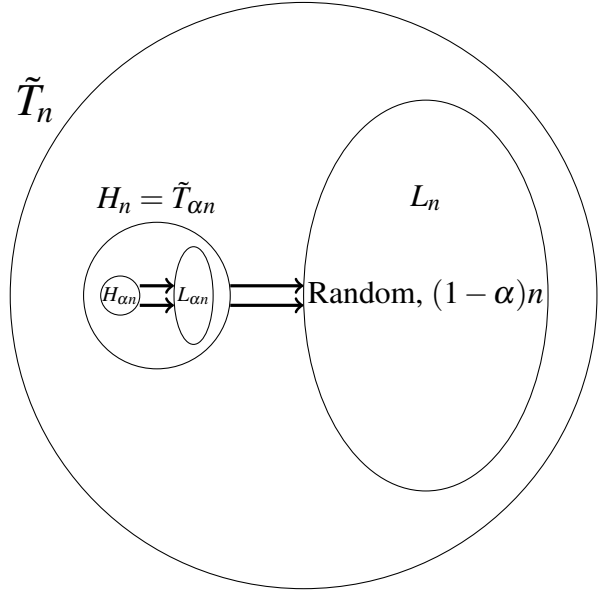
\begin{figure}
\begin{center}
\begin{tikzpicture}[scale=0.65]
    \draw (0,0) circle (6cm);
    \node (Tn) at (-5.5,4) {\Large $\tilde{T}_n$};
	\draw (-3,0) circle (1.5cm);
	\node (H) at (-3,2) {$H_n = \tilde{T}_{\alpha n}$};
	\draw (-3.75,0) circle (.4cm); 
	\node (han) at (-3.75,0) {\scriptsize $H_{\alpha n}$};
	\draw (-2.25,0) ellipse (.4cm and 1cm); \node (lan) at (-2.25,0) {\scriptsize $L_{\alpha n}$};
	\draw[->,very thick] (-3.35,0.2) -- (-2.65,0.2);
	\draw[->,very thick] (-3.35,-.2) -- (-2.65,-.2);
	\draw (2.5,0) ellipse (2.5cm and 4cm);
	\node (QR) at (2.5,2) {$L_n$};
	\node (QR1) at (2.5,0) {Random, $(1-\alpha)n$};
	\draw[->,very thick] (-1.5,0.2) -- (0,0.2);
	\draw[->,very thick] (-1.5,-.2) -- (0,-.2);
\end{tikzpicture}
\end{center}
\caption{A construction asymptotically maximizing the number of copies of  $C_3^+$. For $\alpha \in [0,1]$ and $n$ sufficiently large, this construction $\tilde{T}_n$ can be decomposed into subtournaments $L_n$, of size about $(1-\alpha)n$, and $H_n$, of size about $\alpha n$ with the properties shown above.}\label{fig:iter}
\end{figure}

In this construction, all copies of $C_3^+$ lie completely in $H_n$, completely in $L_n$, or have exactly one vertex in $H_n$ and three vertices forming a $C_3$ in $L_n$. 
Notice that $i(C_3,R_{\lfloor(1-\alpha)n\rfloor}) = 1/4+o(1)$ and $i(C_3^+,R_{\lfloor(1-\alpha)n\rfloor}) = 1/8+o(1)$.
As $i(C_3^+,T_{\lceil \alpha n\rceil})=i(C_3^+,T_n)+o(1)$, we have
\begin{align*}
i(C_3^+,T_n)&=\alpha^4i(C_3^+,T_n)+4\alpha(1-\alpha)^3i(C_3,R_{\lfloor(1-\alpha)n\rfloor})+(1-\alpha)^4i(C_3^+,R_{\lfloor(1-\alpha)n\rfloor})
+o(1)\\
&=\alpha^4i(C_3^+,T_n)+4\alpha(1-\alpha)^3\tfrac14+(1-\alpha)^4\tfrac18+o(1),
\end{align*}
so
\[
i(C_3^+,T_n)=\frac{\alpha(1-\alpha)^3+\tfrac18(1-\alpha)^4}{1-\alpha^4}+o(1). 
\]
Maximizing this quantity gives us
$\alpha = \frac{1}{5}(2\sqrt[3]{9}-2 - \sqrt[3]{3} ) \approx 0.1435836$,
and
\[
i(C_3^+,T_n)+o(1)=\frac18\left( 8-9\sqrt[3]{3}+3\sqrt[3]{9}\right)\approx 0.157500667.
\]
We show in Section~\ref{sec:SR4} 
that all large extremal tournaments for $C_3^+$ essentially look this way. While we can not determine exactly the extremal tournaments $T_n$, we can at least say that the limit object is a unique graphon (when the original definition of graphons is transferred to the tournament setting). 

\begin{theorem}\label{SR4}
Let $(T_n)_{n=1}^\infty$ be a sequence of tournaments on $n$ vertices with $I(C_3^+,T_n)=I(C_3^+,n)$. Let $\alpha = \frac{1}{5}(2\sqrt[3]{9}-2 - \sqrt[3]{3} ) $.
For sufficiently large $n$, the vertex set of $T_n$ can be partitioned into sets $L_n$ and $H_n$ so that 
$|H_n|=\alpha n+o(1)$,
all arcs between these sets are from $H_n$ to $L_n$, the sequence of tournaments $(T_n[L_n])_{n=1}^\infty$ is quasi-random, and $I(C_3^+,T_n[H_n]) = I(C_3^+,|H_n|)$. Hence
\[
i(C_3^+) = \frac18\left( 8-9\sqrt[3]{3}+3\sqrt[3]{9}\right)\approx 0.157500667.
\]
\end{theorem} 
While $C_3^+$ may not be the most interesting tournament to consider at first, we find this extremal construction fascinating. It combines quasi-random parts with iterated blow-ups, and is thus more complex than most known extremal constructions for other problems.

\section{Proof of Theorem~\ref{TC4_Full}}\label{sec:TC4}

\begin{proof}[Proof of Theorem~\ref{TC4_Full}]
We begin by observing the following identity for all tournaments $T$ on at least $4$ vertices: 
\begin{align}\label{eqTC}
i(C_3,T)=\tfrac12 i(C_4,T)+\tfrac14 i(C_3^+,T)+\tfrac14 i(C_3^-,T).
\end{align}
This follows from the fact that the probability to find a $C_3$ when picking three vertices at random is equal to the probability to first find $C_4$, $C_3^+$, or $C_3^-$ when picking four vertices, times the appropriate probability that removing one of these vertices leaves a $C_3$.

Multiplying both sides by $\binom{n}{4}$, we can express this relationship in terms of a direct count of induced $C_4$ for any tournament $T$:
\begin{align*}
I(C_3,T)\cdot\tfrac{n-3}{4}
	&= \tfrac{1}{2}I(C_4,T) + \tfrac{1}{4}(I(C_3^+,T)+I(C_3^-,T)),
\end{align*}implying that \begin{align*}
I(C_4,T) &= I(C_3,T)\cdot\tfrac{n-3}{2} - \tfrac{1}{2}(I(C_3^+,T) + I(C_3^-,T)). 
\end{align*} 

Let $T\in \C_n$. Then every vertex in $T$ has out-degree in $\{\frac{n-2}{2},\frac{n-1}{2},\frac{n}{2}\}$, so by Proposition~\ref{C3}, $I(C_3,T)$ is maximized. On the other hand, the out-neighborhoods and in-neighborhoods of all vertices in $T$ induce transitive tournaments, so $I(C_3^+,T)=I(C_3^-,T)=0$. This shows that $T$ maximizes $I(C_4,T)$. The ideas up to this point are very similar to the proofs in~\cite{MR3549509} and~\cite{bozyk2020inducibility}.

To extend their result to Theorem~\ref{TC4_Full}, it remains to show that no other tournament shares this property.
For this, let $T$ be any $\{C_3^+,C_3^-\}$-free, (near) regular tournament, and let $v_1 \in V(T)$ with $d^+(v_1)=k\in\{\frac{n-2}{2},\frac{n-1}{2},\frac{n}{2}\}$. As $T$ is $C_3^+$-free, the out-neighborhood of $v_1$ is $C_3$-free and therefore transitive, and we may relabel the out-neighbors in this induced order as $\{v_2,v_3,\ldots,v_{k+1}\}$. Similarly, the in-neighborhood is transitive, and we may relabel it in the induced order as $\{v_{k+2},\ldots,v_n\}$.

Now suppose, for the sake of contradiction, that $T\notin\C_n$, and thus there exists an arc $v_iv_j$ with $0<i-j<\frac{n}{2}$ or if $-n<i-j<-\frac{n}{2}$. Let us first assume that $0<i-j<\frac{n}{2}$. As $\{v_2,v_3,\ldots,v_{k+1}\}$ and $\{v_{k+2},\ldots,v_n\}$ are transitively ordered, we have $j\le k$ and $i\ge k+1$. As $v_j$ has out-degree at least $\frac{n-1}{2}$, $v_j$ has an out-neighbor $v_{i'}$ with $i'>i$, implying that $v_iv_{i'}\in E(T)$. But now $T[v_1,v_j,v_i,v_{i'}]\simeq C_3^+$, a contradiction.

Let us now assume that $-n<i-j<-\frac{n}{2}$, and so $i\le k$ and $j>k$. Similarly as before, there now exists a $j'$ with $k<j'<j$ and $v_{j'}v_i\in E$, which again implies that $T[v_1,v_i,v_{j'},v_j]\simeq C_3^+$, a contradiction proving the theorem.
\end{proof}

\section{Proof of Theorem~\ref{SR4}}\label{sec:SR4}

\begin{proof}[Proof of Theorem~\ref{SR4}]
We start with an upper bound for the inducibility of $C_3^+$ using standard flag algebra methods. Notice that the upper bound is not sharp, which is common for extremal constructions involving iterations. We will always assume that $n$ is large enough that we are allowed to suppress lower order terms in our computations.

\begin{claim}
$i(C_3^+,n)\in (0.157500667, 0.157500672)$.
\end{claim}

\noindent
\begin{proof}
We know that $i(C_3^+,n)>0.157500667$ by our construction. Using standard plain flag algebra techniques, we find that
\[
i(C_3^+) \leq \frac{2205009402351719189231861082004597041252856}{14000000000000000000000000000000000000000000}  <  0.157500672.
\]
We computed with flags of size $8$, and the computation, including the subsequent claims, ran for about 6 hours on a standard desktop.
Certificates are too large to be presented here, and do not add much insight. They can be found at \oururl.
\end{proof}

In the next claim, a symmetrization argument gives that every vertex is in roughly the same number of $C_3^+$.
Recall that in the theorem statement, $(T_n)_{n=1}^\infty$  is a sequence of tournaments on $n$ vertices with $I(C_3^+,T_n)=I(C_3^+,n)$. 

\begin{claim}\label{sym}
Every vertex is in $i(C_3^+,T_n){n-1\choose 3}+O(n^2)$ many copies of $C_3^+$.
\end{claim}
\begin{proof}
By definition, the average number of copies a vertex is in is $i(C_3^+,T_n){n-1\choose 3}$.
Let $v$ be a vertex which is in the fewest copies $C_3^+(v)$, and let $w$ be a vertex which is in the most copies $C_3^+(w)$. Let $C_3^+(vw)$ be the number of copies containing both $v$ and $w$.
If we delete $v$, and add a copy of $w$, we gain
\[
C_3^+(w)-C_3^+(v)-C_3^+(vw)
\]
copies of $C_3^+$. As $T_n$ is extremal, this quantity must be non-positive. Observing that $C_3^+(vw)=O(n^2)$ shows the claim. 
\end{proof}

The traditional way to extract structure from flag algebra computations 
 is to look for subgraphs for which the computations tell you that they have zero ($+o(1)$) density in every extremal construction. But this only works if the bounds from the computation are sharp. If the computations do not give sharp bounds like in our case, another approach is to do the opposite, and to compute bounds on subgraphs which occur with high density to find a general structure of the extremal example, and then use stability methods to establish the finer structure. Neither of these approaches has much promise in this problem without new ideas. As a large part of the conjectured extremal tournament is quasi-random, all subgraphs appear with a frequency similar to a random tournament, and structural differences to a random tournament are difficult to extract.

Inspired by the conjectured extremal tournament, we are looking for other features. A first observation is that the degree distribution is concentrated around a few discrete values. All vertices in $L_n$ have about the same fairly small out-degree, and all vertices in $H_n$ have very large out-degree, where the corresponding discrete values are a sequence converging to $1$ when normalized. A second observation is that all arcs between $L_n$ and $H_n$ are directed from $H_n$ to $L_n$. We use flag algebra computations to prove that these two observations are true in every extremal tournament, and from this we are able to prove the theorem.

Let $deg(x)$ be the normalized out-degree distribution function for an extremal tournament $T_n$: 
\[
deg(x) = \tfrac{1}{n}\abs{\{v \in V(T) : d^+(v)=xn\}}.
\]
For the remainder of the proof, the word ``normalized'' will be suppressed for simplicity. 
To make our computations more intuitive to follow, we will often denote the quantity $i(H,T_n)$ by a picture of the graph $H$, so we might write 
\[
\begin{tikzpicture}[baseline=1.3ex,scale=0.65]
			\node[shape=circle,draw=black,fill=black,inner sep=2pt] (1) at (0,0) {};
			\node[shape=circle,draw=black,fill=black,inner sep=2pt] (2) at (1,0) {};
			\node[shape=circle,draw=black,fill=black,inner sep=2pt] (3) at (1,1) {};
			\node[shape=circle,draw=black,fill=black,inner sep=2pt] (4) at (0,1) {};
			\draw[dedge] (1) -> (2);
			\draw[dedge] (2) -> (3);
			\draw[dedge] (3) -> (4);
			\draw[dedge] (4) -> (2);
			\draw[dedge] (1) -> (3);
			\draw[dedge] (1) -> (4);
		\end{tikzpicture}
=
i(C_3^+,T_n).
\]
We now show that $T_n$ has a degree distribution similar to the conjectured example, i.e. almost all vertices have degrees in small intervals around the degrees appearing in the construction. While we might be able to separate the high degree vertices into more degree bands with more effort, this will not be needed later, so we collect them all in one interval. We state these bounds up to a precision useful later in the proof.

\begin{claim}\label{support}
For all $x \in [0,0.416] \cup [0.44057,0.8849]$, $deg(x) = 0$. In other words, all but $o(n)$ vertices have degrees either in $(0.416,0.44057)$ or in $(0.8849,1]$.
\end{claim}

\noindent
\begin{proof}
We prove this claim by showing three bounds.
First we investigate vertices $v$ with $d^+(v) \leq 0.85n$ and
obtain lower and upper bounds on $d^+(v)$, namely that $d^+(v) \in (0.416,0.44057)$. For the third bound, we switch to vertices $v$ with $d^+(v) \geq 0.85n$ and show that actually $d^+(v) > 0.8849n$.

 We begin with the lower bound of the support of $deg(x)$. Fix some vertex $v \in V(T_n)$ and color all vertices in $N^+(v)$ black and color $N^-(v)$ white. We will use flag algebras to bound the proportion of black vertices in $T_n-v$, and to this end we begin setting up a program that can be bounded by the plain flag algebra method. Since $i(C_3^+,T_n-v) >0.157500667$, we know the sum of the densities of all 2-colorings of $C_3^+$ is at least 0.157500667. 
We reduce our search space with the constraint that
$\begin{tikzpicture}[scale=0.65]
	\node[shape=circle,draw=black,fill=black,inner sep=2pt] (1) at (0,0){};
\end{tikzpicture} \leq 0.85$, interpreted as $v$ having normalized out-degree at most 0.85.

Ignoring lower order terms, we also know that every vertex is in the same number of $C_3^+$ (see Claim \ref{sym}), so we can add an additional constraint to reflect this fact. If $v$ plays the role of the source vertex in the $C_3^+$, then the remaining three vertices are all in $N^+(v)$ and induce a $C_3$. Otherwise, $v$ plays the role of one of the vertices in the $C_3$, and the other three vertices induce a transitive triangle where the source and sink are in $N^-(v)$ and the last vertex is in $N^+(x)$. Our coloring scheme thus allows us to include the final bound in the following program:

\begin{quote} 
	Objective: \begin{quote}
		minimize $\begin{tikzpicture}[scale=0.65]
	\node[shape=circle,draw=black,fill=black,inner sep=2pt] (1) at (0,0){};
\end{tikzpicture}$
		\end{quote}
	Constraints: \begin{quote}
		$\begin{tikzpicture}[scale=0.65]
	\node[shape=circle,draw=black,fill=black,inner sep=2pt] (1) at (0,0){};
\end{tikzpicture} \leq 0.85$ \\
		$0.157500667 \leq 
		\begin{tikzpicture}[baseline=1.3ex,scale=0.65]
			\node[shape=circle,draw=black,fill=black,inner sep=2pt] (1) at (0,0){};
			\node[shape=circle,draw=black,fill=black,inner sep=2pt] (2) at (0,1){};
			\node[shape=circle,draw=black,fill=black,inner sep=2pt] (3) at (1,1){};
			\node[shape=circle,draw=black,fill=black,inner sep=2pt] (4) at (1,0){};
			\draw (1) -> (2) -> (3) -> (4) -> (2) (1) -> (3) (1) -> (4);
			\draw[dedge] (1) -> (2);
			\draw[dedge] (2) -> (3);
			\draw[dedge] (3) -> (4);
			\draw[dedge] (4) -> (2);
			\draw[dedge] (1) -> (3);
			\draw[dedge] (1) -> (4);
		\end{tikzpicture}
		+ \begin{tikzpicture}[baseline=1.3ex,scale=0.65]
			\node[shape=circle,draw=black,fill=black,inner sep=2pt] (1) at (0,0){};
			\node[shape=circle,draw=black,fill=black,inner sep=2pt] (2) at (0,1){};
			\node[shape=circle,draw=black,fill=black,inner sep=2pt] (3) at (1,1){};
			\node[shape=circle,draw=black,fill=white,inner sep=2pt] (4) at (1,0){};
			\draw (1) -> (2) -> (3) -> (4) -> (2) (1) -> (3) (1) -> (4);	
			\draw[dedge] (1) -> (2);
			\draw[dedge] (2) -> (3);
			\draw[dedge] (3) -> (4);
			\draw[dedge] (4) -> (2);
			\draw[dedge] (1) -> (3);
			\draw[dedge] (1) -> (4);	
		\end{tikzpicture}
		+ \begin{tikzpicture}[baseline=1.3ex,scale=0.65]
			\node[shape=circle,draw=black,fill=black,inner sep=2pt] (1) at (0,0){};
			\node[shape=circle,draw=black,fill=black,inner sep=2pt] (2) at (0,1){};
			\node[shape=circle,draw=black,fill=white,inner sep=2pt] (3) at (1,1){};
			\node[shape=circle,draw=black,fill=white,inner sep=2pt] (4) at (1,0){};
			\draw (1) -> (2) -> (3) -> (4) -> (2) (1) -> (3) (1) -> (4);		
			\draw[dedge] (1) -> (2);
			\draw[dedge] (2) -> (3);
			\draw[dedge] (3) -> (4);
			\draw[dedge] (4) -> (2);
			\draw[dedge] (1) -> (3);
			\draw[dedge] (1) -> (4);
		\end{tikzpicture}
		+ \begin{tikzpicture}[baseline=1.3ex,scale=0.65]
			\node[shape=circle,draw=black,fill=black,inner sep=2pt] (1) at (0,0){};
			\node[shape=circle,draw=black,fill=white,inner sep=2pt] (2) at (0,1){};
			\node[shape=circle,draw=black,fill=white,inner sep=2pt] (3) at (1,1){};
			\node[shape=circle,draw=black,fill=white,inner sep=2pt] (4) at (1,0){};
			\draw (1) -> (2) -> (3) -> (4) -> (2) (1) -> (3) (1) -> (4);	
			\draw[dedge] (1) -> (2);
			\draw[dedge] (2) -> (3);
			\draw[dedge] (3) -> (4);
			\draw[dedge] (4) -> (2);
			\draw[dedge] (1) -> (3);
			\draw[dedge] (1) -> (4);	
		\end{tikzpicture}
		+ \begin{tikzpicture}[baseline=1.3ex,scale=0.65]
			\node[shape=circle,draw=black,fill=white,inner sep=2pt] (1) at (0,0){};
			\node[shape=circle,draw=black,fill=black,inner sep=2pt] (2) at (0,1){};
			\node[shape=circle,draw=black,fill=black,inner sep=2pt] (3) at (1,1){};
			\node[shape=circle,draw=black,fill=black,inner sep=2pt] (4) at (1,0){};
			\draw (1) -> (2) -> (3) -> (4) -> (2) (1) -> (3) (1) -> (4);		
			\draw[dedge] (1) -> (2);
			\draw[dedge] (2) -> (3);
			\draw[dedge] (3) -> (4);
			\draw[dedge] (4) -> (2);
			\draw[dedge] (1) -> (3);
			\draw[dedge] (1) -> (4);
		\end{tikzpicture}
		+ \begin{tikzpicture}[baseline=1.3ex,scale=0.65]
			\node[shape=circle,draw=black,fill=white,inner sep=2pt] (1) at (0,0){};
			\node[shape=circle,draw=black,fill=black,inner sep=2pt] (2) at (0,1){};
			\node[shape=circle,draw=black,fill=black,inner sep=2pt] (3) at (1,1){};
			\node[shape=circle,draw=black,fill=white,inner sep=2pt] (4) at (1,0){};
			\draw (1) -> (2) -> (3) -> (4) -> (2) (1) -> (3) (1) -> (4);		
			\draw[dedge] (1) -> (2);
			\draw[dedge] (2) -> (3);
			\draw[dedge] (3) -> (4);
			\draw[dedge] (4) -> (2);
			\draw[dedge] (1) -> (3);
			\draw[dedge] (1) -> (4);
		\end{tikzpicture}
		+ \begin{tikzpicture}[baseline=1.3ex,scale=0.65]
			\node[shape=circle,draw=black,fill=white,inner sep=2pt] (1) at (0,0){};
			\node[shape=circle,draw=black,fill=black,inner sep=2pt] (2) at (0,1){};
			\node[shape=circle,draw=black,fill=white,inner sep=2pt] (3) at (1,1){};
			\node[shape=circle,draw=black,fill=white,inner sep=2pt] (4) at (1,0){};
			\draw (1) -> (2) -> (3) -> (4) -> (2) (1) -> (3) (1) -> (4);		
			\draw[dedge] (1) -> (2);
			\draw[dedge] (2) -> (3);
			\draw[dedge] (3) -> (4);
			\draw[dedge] (4) -> (2);
			\draw[dedge] (1) -> (3);
			\draw[dedge] (1) -> (4);
		\end{tikzpicture}
		+ \begin{tikzpicture}[baseline=1.3ex,scale=0.65]
			\node[shape=circle,draw=black,fill=white,inner sep=2pt] (1) at (0,0){};
			\node[shape=circle,draw=black,fill=white,inner sep=2pt] (2) at (0,1){};
			\node[shape=circle,draw=black,fill=white,inner sep=2pt] (3) at (1,1){};
			\node[shape=circle,draw=black,fill=white,inner sep=2pt] (4) at (1,0){};
			\draw (1) -> (2) -> (3) -> (4) -> (2) (1) -> (3) (1) -> (4);		
			\draw[dedge] (1) -> (2);
			\draw[dedge] (2) -> (3);
			\draw[dedge] (3) -> (4);
			\draw[dedge] (4) -> (2);
			\draw[dedge] (1) -> (3);
			\draw[dedge] (1) -> (4);
		\end{tikzpicture}$ \\
		$0.157500667 \leq \begin{tikzpicture}[baseline=1.3ex,scale=0.65]
			\node[shape=circle,draw=black,fill=black,inner sep=2pt] (1) at (0,0){};
			\node[shape=circle,draw=black,fill=black,inner sep=2pt] (2) at (0.5,1){};
			\node[shape=circle,draw=black,fill=black,inner sep=2pt] (3) at (1,0){};
			\draw[dedge] (1) -> (2);
			\draw[dedge] (2) -> (3);
			\draw[dedge] (3) -> (1);
		\end{tikzpicture}
		+ \begin{tikzpicture}[baseline=1.3ex,scale=0.65]
			\node[shape=circle,draw=black,fill=white,inner sep=2pt] (1) at (0,0){};
			\node[shape=circle,draw=black,fill=black,inner sep=2pt] (2) at (0.5,1) {};
			\node[shape=circle,draw=black,fill=white,inner sep=2pt] (3) at (1,0) {};
			\draw[dedge] (1) -> (2);
			\draw[dedge] (2) -> (3);
			\draw[dedge] (1) -> (3);
		\end{tikzpicture}
		$
		\end{quote}
\end{quote}

From this program, we find that $\begin{tikzpicture}[scale=0.65]
	\node[shape=circle,draw=black,fill=black,inner sep=2pt] (1) at (0,0){};
\end{tikzpicture} > 0.416$. More precisely,
\[
\begin{tikzpicture}[scale=0.65]
	\node[shape=circle,fill=black,inner sep=2pt] (1) at (0,0){};
\end{tikzpicture} \geq \frac{3745132053776853970635292684882343419187823353945331}{ 9000000000000000000000000000000000000000000000000000}  >   0.416.
\]

Similarly, we obtain $\begin{tikzpicture}[scale=0.65]
	\node[shape=circle,draw=black,fill=black,inner sep=2pt] (1) at (0,0){};
\end{tikzpicture} < 0.44057$, or more precisely that
\[
\begin{tikzpicture}[scale=0.65]
	\node[shape=circle,fill=black,inner sep=2pt] (1) at (0,0){};
\end{tikzpicture} \leq \frac{3965045776267019146300369491791058488457641019475359}{9000000000000000000000000000000000000000000000000000}  < 0.44057,
\] 
from the following program:

\begin{quote} 
	Objective: \begin{quote}
		maximize $\begin{tikzpicture}[scale=0.65]
	\node[shape=circle,draw=black,fill=black,inner sep=2pt] (1) at (0,0){};
\end{tikzpicture}$
		\end{quote}
	Constraints: \begin{quote}
		$\begin{tikzpicture}[scale=0.65]
	\node[shape=circle,draw=black,fill=black,inner sep=2pt] (1) at (0,0){};
\end{tikzpicture} \leq 0.85$ \\
		$0.157500667 \leq 
		\begin{tikzpicture}[baseline=1.3ex,scale=0.65]
			\node[shape=circle,draw=black,fill=black,inner sep=2pt] (1) at (0,0){};
			\node[shape=circle,draw=black,fill=black,inner sep=2pt] (2) at (0,1){};
			\node[shape=circle,draw=black,fill=black,inner sep=2pt] (3) at (1,1){};
			\node[shape=circle,draw=black,fill=black,inner sep=2pt] (4) at (1,0){};
			\draw (1) -> (2) -> (3) -> (4) -> (2) (1) -> (3) (1) -> (4);
			\draw[dedge] (1) -> (2);
			\draw[dedge] (2) -> (3);
			\draw[dedge] (3) -> (4);
			\draw[dedge] (4) -> (2);
			\draw[dedge] (1) -> (3);
			\draw[dedge] (1) -> (4);
		\end{tikzpicture}
		+ \begin{tikzpicture}[baseline=1.3ex,scale=0.65]
			\node[shape=circle,draw=black,fill=black,inner sep=2pt] (1) at (0,0){};
			\node[shape=circle,draw=black,fill=black,inner sep=2pt] (2) at (0,1){};
			\node[shape=circle,draw=black,fill=black,inner sep=2pt] (3) at (1,1){};
			\node[shape=circle,draw=black,fill=white,inner sep=2pt] (4) at (1,0){};
			\draw (1) -> (2) -> (3) -> (4) -> (2) (1) -> (3) (1) -> (4);	
			\draw[dedge] (1) -> (2);
			\draw[dedge] (2) -> (3);
			\draw[dedge] (3) -> (4);
			\draw[dedge] (4) -> (2);
			\draw[dedge] (1) -> (3);
			\draw[dedge] (1) -> (4);	
		\end{tikzpicture}
		+ \begin{tikzpicture}[baseline=1.3ex,scale=0.65]
			\node[shape=circle,draw=black,fill=black,inner sep=2pt] (1) at (0,0){};
			\node[shape=circle,draw=black,fill=black,inner sep=2pt] (2) at (0,1){};
			\node[shape=circle,draw=black,fill=white,inner sep=2pt] (3) at (1,1){};
			\node[shape=circle,draw=black,fill=white,inner sep=2pt] (4) at (1,0){};
			\draw (1) -> (2) -> (3) -> (4) -> (2) (1) -> (3) (1) -> (4);		
			\draw[dedge] (1) -> (2);
			\draw[dedge] (2) -> (3);
			\draw[dedge] (3) -> (4);
			\draw[dedge] (4) -> (2);
			\draw[dedge] (1) -> (3);
			\draw[dedge] (1) -> (4);
		\end{tikzpicture}
		+ \begin{tikzpicture}[baseline=1.3ex,scale=0.65]
			\node[shape=circle,draw=black,fill=black,inner sep=2pt] (1) at (0,0){};
			\node[shape=circle,draw=black,fill=white,inner sep=2pt] (2) at (0,1){};
			\node[shape=circle,draw=black,fill=white,inner sep=2pt] (3) at (1,1){};
			\node[shape=circle,draw=black,fill=white,inner sep=2pt] (4) at (1,0){};
			\draw (1) -> (2) -> (3) -> (4) -> (2) (1) -> (3) (1) -> (4);	
			\draw[dedge] (1) -> (2);
			\draw[dedge] (2) -> (3);
			\draw[dedge] (3) -> (4);
			\draw[dedge] (4) -> (2);
			\draw[dedge] (1) -> (3);
			\draw[dedge] (1) -> (4);	
		\end{tikzpicture}
		+ \begin{tikzpicture}[baseline=1.3ex,scale=0.65]
			\node[shape=circle,draw=black,fill=white,inner sep=2pt] (1) at (0,0){};
			\node[shape=circle,draw=black,fill=black,inner sep=2pt] (2) at (0,1){};
			\node[shape=circle,draw=black,fill=black,inner sep=2pt] (3) at (1,1){};
			\node[shape=circle,draw=black,fill=black,inner sep=2pt] (4) at (1,0){};
			\draw (1) -> (2) -> (3) -> (4) -> (2) (1) -> (3) (1) -> (4);		
			\draw[dedge] (1) -> (2);
			\draw[dedge] (2) -> (3);
			\draw[dedge] (3) -> (4);
			\draw[dedge] (4) -> (2);
			\draw[dedge] (1) -> (3);
			\draw[dedge] (1) -> (4);
		\end{tikzpicture}
		+ \begin{tikzpicture}[baseline=1.3ex,scale=0.65]
			\node[shape=circle,draw=black,fill=white,inner sep=2pt] (1) at (0,0){};
			\node[shape=circle,draw=black,fill=black,inner sep=2pt] (2) at (0,1){};
			\node[shape=circle,draw=black,fill=black,inner sep=2pt] (3) at (1,1){};
			\node[shape=circle,draw=black,fill=white,inner sep=2pt] (4) at (1,0){};
			\draw (1) -> (2) -> (3) -> (4) -> (2) (1) -> (3) (1) -> (4);		
			\draw[dedge] (1) -> (2);
			\draw[dedge] (2) -> (3);
			\draw[dedge] (3) -> (4);
			\draw[dedge] (4) -> (2);
			\draw[dedge] (1) -> (3);
			\draw[dedge] (1) -> (4);
		\end{tikzpicture}
		+ \begin{tikzpicture}[baseline=1.3ex,scale=0.65]
			\node[shape=circle,draw=black,fill=white,inner sep=2pt] (1) at (0,0){};
			\node[shape=circle,draw=black,fill=black,inner sep=2pt] (2) at (0,1){};
			\node[shape=circle,draw=black,fill=white,inner sep=2pt] (3) at (1,1){};
			\node[shape=circle,draw=black,fill=white,inner sep=2pt] (4) at (1,0){};
			\draw (1) -> (2) -> (3) -> (4) -> (2) (1) -> (3) (1) -> (4);		
			\draw[dedge] (1) -> (2);
			\draw[dedge] (2) -> (3);
			\draw[dedge] (3) -> (4);
			\draw[dedge] (4) -> (2);
			\draw[dedge] (1) -> (3);
			\draw[dedge] (1) -> (4);
		\end{tikzpicture}
		+ \begin{tikzpicture}[baseline=1.3ex,scale=0.65]
			\node[shape=circle,draw=black,fill=white,inner sep=2pt] (1) at (0,0){};
			\node[shape=circle,draw=black,fill=white,inner sep=2pt] (2) at (0,1){};
			\node[shape=circle,draw=black,fill=white,inner sep=2pt] (3) at (1,1){};
			\node[shape=circle,draw=black,fill=white,inner sep=2pt] (4) at (1,0){};
			\draw (1) -> (2) -> (3) -> (4) -> (2) (1) -> (3) (1) -> (4);		
			\draw[dedge] (1) -> (2);
			\draw[dedge] (2) -> (3);
			\draw[dedge] (3) -> (4);
			\draw[dedge] (4) -> (2);
			\draw[dedge] (1) -> (3);
			\draw[dedge] (1) -> (4);
		\end{tikzpicture}$ \\
		$0.157500667 \leq \begin{tikzpicture}[baseline=1.3ex,scale=0.65]
			\node[shape=circle,draw=black,fill=black,inner sep=2pt] (1) at (0,0){};
			\node[shape=circle,draw=black,fill=black,inner sep=2pt] (2) at (0.5,1){};
			\node[shape=circle,draw=black,fill=black,inner sep=2pt] (3) at (1,0){};
			\draw[dedge] (1) -> (2);
			\draw[dedge] (2) -> (3);
			\draw[dedge] (3) -> (1);
		\end{tikzpicture}
		+ \begin{tikzpicture}[baseline=1.3ex,scale=0.65]
			\node[shape=circle,draw=black,fill=white,inner sep=2pt] (1) at (0,0){};
			\node[shape=circle,draw=black,fill=black,inner sep=2pt] (2) at (0.5,1) {};
			\node[shape=circle,draw=black,fill=white,inner sep=2pt] (3) at (1,0) {};
			\draw[dedge] (1) -> (2);
			\draw[dedge] (2) -> (3);
			\draw[dedge] (1) -> (3);
		\end{tikzpicture}
		$
		\end{quote}
\end{quote}

These two results imply that for large enough $n$, no vertices have normalized out-degree in $[0,0.416] \cup [0.44057,0.85]$.
We extend this result with the following program restricting the degree of large out-degree vertices: 

\begin{quote} 
	Objective: \begin{quote}
		minimize $\begin{tikzpicture}[scale=0.65]
	\node[shape=circle,draw=black,fill=black,inner sep=2pt] (1) at (0,0){};
\end{tikzpicture}$
		\end{quote}
	Constraints: \begin{quote}
		$\begin{tikzpicture}[scale=0.65]
	\node[shape=circle,draw=black,fill=black,inner sep=2pt] (1) at (0,0){};
\end{tikzpicture} \geq 0.85$ \\
		$0.157500667 \leq 
		\begin{tikzpicture}[baseline=1.3ex,scale=0.65]
			\node[shape=circle,draw=black,fill=black,inner sep=2pt] (1) at (0,0){};
			\node[shape=circle,draw=black,fill=black,inner sep=2pt] (2) at (0,1){};
			\node[shape=circle,draw=black,fill=black,inner sep=2pt] (3) at (1,1){};
			\node[shape=circle,draw=black,fill=black,inner sep=2pt] (4) at (1,0){};
			\draw (1) -> (2) -> (3) -> (4) -> (2) (1) -> (3) (1) -> (4);
			\draw[dedge] (1) -> (2);
			\draw[dedge] (2) -> (3);
			\draw[dedge] (3) -> (4);
			\draw[dedge] (4) -> (2);
			\draw[dedge] (1) -> (3);
			\draw[dedge] (1) -> (4);
		\end{tikzpicture}
		+ \begin{tikzpicture}[baseline=1.3ex,scale=0.65]
			\node[shape=circle,draw=black,fill=black,inner sep=2pt] (1) at (0,0){};
			\node[shape=circle,draw=black,fill=black,inner sep=2pt] (2) at (0,1){};
			\node[shape=circle,draw=black,fill=black,inner sep=2pt] (3) at (1,1){};
			\node[shape=circle,draw=black,fill=white,inner sep=2pt] (4) at (1,0){};
			\draw (1) -> (2) -> (3) -> (4) -> (2) (1) -> (3) (1) -> (4);	
			\draw[dedge] (1) -> (2);
			\draw[dedge] (2) -> (3);
			\draw[dedge] (3) -> (4);
			\draw[dedge] (4) -> (2);
			\draw[dedge] (1) -> (3);
			\draw[dedge] (1) -> (4);	
		\end{tikzpicture}
		+ \begin{tikzpicture}[baseline=1.3ex,scale=0.65]
			\node[shape=circle,draw=black,fill=black,inner sep=2pt] (1) at (0,0){};
			\node[shape=circle,draw=black,fill=black,inner sep=2pt] (2) at (0,1){};
			\node[shape=circle,draw=black,fill=white,inner sep=2pt] (3) at (1,1){};
			\node[shape=circle,draw=black,fill=white,inner sep=2pt] (4) at (1,0){};
			\draw (1) -> (2) -> (3) -> (4) -> (2) (1) -> (3) (1) -> (4);		
			\draw[dedge] (1) -> (2);
			\draw[dedge] (2) -> (3);
			\draw[dedge] (3) -> (4);
			\draw[dedge] (4) -> (2);
			\draw[dedge] (1) -> (3);
			\draw[dedge] (1) -> (4);
		\end{tikzpicture}
		+ \begin{tikzpicture}[baseline=1.3ex,scale=0.65]
			\node[shape=circle,draw=black,fill=black,inner sep=2pt] (1) at (0,0){};
			\node[shape=circle,draw=black,fill=white,inner sep=2pt] (2) at (0,1){};
			\node[shape=circle,draw=black,fill=white,inner sep=2pt] (3) at (1,1){};
			\node[shape=circle,draw=black,fill=white,inner sep=2pt] (4) at (1,0){};
			\draw (1) -> (2) -> (3) -> (4) -> (2) (1) -> (3) (1) -> (4);	
			\draw[dedge] (1) -> (2);
			\draw[dedge] (2) -> (3);
			\draw[dedge] (3) -> (4);
			\draw[dedge] (4) -> (2);
			\draw[dedge] (1) -> (3);
			\draw[dedge] (1) -> (4);	
		\end{tikzpicture}
		+ \begin{tikzpicture}[baseline=1.3ex,scale=0.65]
			\node[shape=circle,draw=black,fill=white,inner sep=2pt] (1) at (0,0){};
			\node[shape=circle,draw=black,fill=black,inner sep=2pt] (2) at (0,1){};
			\node[shape=circle,draw=black,fill=black,inner sep=2pt] (3) at (1,1){};
			\node[shape=circle,draw=black,fill=black,inner sep=2pt] (4) at (1,0){};
			\draw (1) -> (2) -> (3) -> (4) -> (2) (1) -> (3) (1) -> (4);		
			\draw[dedge] (1) -> (2);
			\draw[dedge] (2) -> (3);
			\draw[dedge] (3) -> (4);
			\draw[dedge] (4) -> (2);
			\draw[dedge] (1) -> (3);
			\draw[dedge] (1) -> (4);
		\end{tikzpicture}
		+ \begin{tikzpicture}[baseline=1.3ex,scale=0.65]
			\node[shape=circle,draw=black,fill=white,inner sep=2pt] (1) at (0,0){};
			\node[shape=circle,draw=black,fill=black,inner sep=2pt] (2) at (0,1){};
			\node[shape=circle,draw=black,fill=black,inner sep=2pt] (3) at (1,1){};
			\node[shape=circle,draw=black,fill=white,inner sep=2pt] (4) at (1,0){};
			\draw (1) -> (2) -> (3) -> (4) -> (2) (1) -> (3) (1) -> (4);		
			\draw[dedge] (1) -> (2);
			\draw[dedge] (2) -> (3);
			\draw[dedge] (3) -> (4);
			\draw[dedge] (4) -> (2);
			\draw[dedge] (1) -> (3);
			\draw[dedge] (1) -> (4);
		\end{tikzpicture}
		+ \begin{tikzpicture}[baseline=1.3ex,scale=0.65]
			\node[shape=circle,draw=black,fill=white,inner sep=2pt] (1) at (0,0){};
			\node[shape=circle,draw=black,fill=black,inner sep=2pt] (2) at (0,1){};
			\node[shape=circle,draw=black,fill=white,inner sep=2pt] (3) at (1,1){};
			\node[shape=circle,draw=black,fill=white,inner sep=2pt] (4) at (1,0){};
			\draw (1) -> (2) -> (3) -> (4) -> (2) (1) -> (3) (1) -> (4);		
			\draw[dedge] (1) -> (2);
			\draw[dedge] (2) -> (3);
			\draw[dedge] (3) -> (4);
			\draw[dedge] (4) -> (2);
			\draw[dedge] (1) -> (3);
			\draw[dedge] (1) -> (4);
		\end{tikzpicture}
		+ \begin{tikzpicture}[baseline=1.3ex,scale=0.65]
			\node[shape=circle,draw=black,fill=white,inner sep=2pt] (1) at (0,0){};
			\node[shape=circle,draw=black,fill=white,inner sep=2pt] (2) at (0,1){};
			\node[shape=circle,draw=black,fill=white,inner sep=2pt] (3) at (1,1){};
			\node[shape=circle,draw=black,fill=white,inner sep=2pt] (4) at (1,0){};
			\draw (1) -> (2) -> (3) -> (4) -> (2) (1) -> (3) (1) -> (4);		
			\draw[dedge] (1) -> (2);
			\draw[dedge] (2) -> (3);
			\draw[dedge] (3) -> (4);
			\draw[dedge] (4) -> (2);
			\draw[dedge] (1) -> (3);
			\draw[dedge] (1) -> (4);
		\end{tikzpicture}$ \\
		$0.157500667 \leq \begin{tikzpicture}[baseline=1.3ex,scale=0.65]
			\node[shape=circle,draw=black,fill=black,inner sep=2pt] (1) at (0,0){};
			\node[shape=circle,draw=black,fill=black,inner sep=2pt] (2) at (0.5,1){};
			\node[shape=circle,draw=black,fill=black,inner sep=2pt] (3) at (1,0){};
			\draw[dedge] (1) -> (2);
			\draw[dedge] (2) -> (3);
			\draw[dedge] (3) -> (1);
		\end{tikzpicture}
		+ \begin{tikzpicture}[baseline=1.3ex,scale=0.65]
			\node[shape=circle,draw=black,fill=white,inner sep=2pt] (1) at (0.5,1){};
			\node[shape=circle,draw=black,fill=black,inner sep=2pt] (2) at (1,0) {};
			\node[shape=circle,draw=black,fill=white,inner sep=2pt] (3) at (0,0) {};
			\draw[dedge] (1) -> (2);
			\draw[dedge] (2) -> (3);
			\draw[dedge] (1) -> (3);
		\end{tikzpicture}
		$
		\end{quote}
\end{quote}

This program outputs the lower bound $\begin{tikzpicture}[scale=0.65]
	\node[shape=circle,fill=black,inner sep=2pt] (1) at (0,0){};
\end{tikzpicture} > 0.8849$, completing the proof of this claim.
More precisely, it gives
\[
\begin{tikzpicture}[scale=0.65]
	\node[shape=circle,draw=black,fill=black,inner sep=2pt] (1) at (0,0){};
\end{tikzpicture} \geq \frac{7964349279220411495203511946962186030422903216282624}{ 9000000000000000000000000000000000000000000000000000} >  0.88492769769. 
\]
All certificates can be found at 
\oururl.
\end{proof}

Let $H_n$ be the set of vertices in $T_n$ with normalized out-degree in $(0.8849,1]$, and $L_n$ be the set of vertices with normalized out-degree in $(0.416,0.44057)$. The above claim implies that $H_n \cup L_n = V(T_n)$. We now show that no arcs in $T_n$ are directed from $L_n$ to $H_n$, once again using a coloring-scheme to acquire localized information in an extremal construction.

\begin{claim}\label{cut}
For every $x \in L_n$ and $y \in H_n$, $yx \in E(T_n)$.
\end{claim}

\noindent
\begin{proof}
Let $x\in L_n$, $y \in H_n$, so $x$ has normalized out degree in $[0.415,0.441]$ and $y$ has normalized out-degree at least 0.8849. We color $V(T_n)-\{x,y\}$ with the following scheme, in which the top color represents the relation to $x$, and the bottom color represents the relation to $y$ (see also Figure~\ref{4colors}): 
\begin{itemize}
	\item Assign color black-black $\OO$~
to $N^+(x) \cap N^+(y)$,
	\item Assign color black-white $\OI$~
to $N^+(x) \cap N^-(y)$,
	\item Assign color white-black $\IO$~
    	 to $N^-(x) \cap N^+(y)$,
	\item Assign color white-white $\II$~
to $N^-(x) \cap N^-(y)$.
\end{itemize}

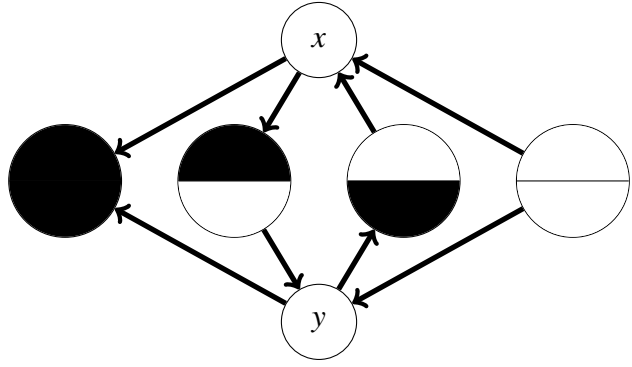
\begin{figure}[h]\centering
\begin{tikzpicture}[scale=.75]
	\node[shape=circle,draw=black,fill=white,minimum size=1cm] (x) at (4.5,2.5) {$x$};
	\node[shape=circle,draw=black,fill=white,minimum size=1cm] (y) at (4.5,-2.5) {$y$};
	\node[shape=circle split, circle split part fill={black,black}, minimum size=1.5cm, draw=black] (OO) at (  0,0) {\,\,\,\,}; 
	\node[shape=circle split, circle split part fill={black,white}, minimum size=1.5cm, draw=black] (OI) at (  3,0) {\,\,\,\,}; 
	\node[shape=circle split, circle split part fill={white,black}, minimum size=1.5cm, draw=black] (IO) at (  6,0) {\,\,\,\,}; 
	\node[shape=circle split, circle split part fill={white,white}, minimum size=1.5cm, draw=black] (II) at (  9,0) {\,\,\,\,};

	\draw[line width=2pt,->] (x) -> (OO);
	\draw[line width=2pt,->] (x) -> (OI);
	\draw[line width=2pt,->] (IO) -> (x);
	\draw[line width=2pt,->] (II) -> (x);
	\draw[line width=2pt,->] (y) -> (OO);
	\draw[line width=2pt,->] (y) -> (IO);
	\draw[line width=2pt,->] (II) -> (y);
	\draw[line width=2pt,->] (OI) -> (y);
\end{tikzpicture}

%
%
%
\caption{Four-coloring scheme for $T_n - \{x,y\}$}.\label{4colors}
\end{figure}

In order to model the out-degree assumptions, we will use the following constraints:
\begin{align*}
0.416 \leq \OO+\OI \leq 0.44057 \text{\hskip 1em and \hskip 1em }
0.8849 \leq \OO+\IO.
\end{align*}
As in the proof of Claim \ref{support}, any programs involving this color scheme can include a constraint to ensure that $x$ is in the right number of $C_3^+$ with vertices in $V(T_n)\backslash y$, and that $y$ is in the right number of $C_3^+$ with vertices in $V(T_n)\backslash x$. 

The purpose of this set up is to show that $x \to y$ results in fewer $C_3^+$ than $y \to x$, so we need to count $C_3^+$ which include both of these vertices, with the edge between the vertices in either direction. For this, we look again at Figure~\ref{4colors}. If $x \to y$, we create a $C_3^+$ with each arc 
$
\begin{tikzpicture}[baseline=-0.5ex]
\node (1) at (0,0) {};
\node (2) at (0.7,0) {};
\node (1') at (0,0) {$\II$};
\node (2') at (0.7,0) {$\IO$};
\draw[dedge] (1) -> (2);
\end{tikzpicture}
$
 and with each arc 
 $
\begin{tikzpicture}[baseline=-0.5ex]
\node (1) at (0,0) {};
\node (2) at (0.7,0) {};
\node (1') at (0,0) {$\OO$};
\node (2') at (0.7,0) {$\OI$};
\draw[dedge] (1) -> (2);
\end{tikzpicture}
$
. On the other hand, if $y \to x$, we create a $C_3^+$ with each arc
 $
\begin{tikzpicture}[baseline=-0.5ex]
\node (1) at (0,0) {};
\node (2) at (0.7,0) {};
\node (1') at (0,0) {$\II$};
\node (2') at (0.7,0) {$\OI$};
\draw[dedge] (1) -> (2);
\end{tikzpicture}
$
 and with each arc 
 $
\begin{tikzpicture}[baseline=-0.5ex]
\node (1) at (0,0) {};
\node (2) at (0.7,0) {};
\node (1') at (0,0) {$\OO$};
\node (2') at (0.7,0) {$\IO$};
\draw[dedge] (1) -> (2);
\end{tikzpicture}
$.

Similarly as above, we can now pose the following program bounding the difference between $C_3^+$ containing $x\to y$ and containing $y\to x$. Note that there are up to $96$ different $C_3^+$ in $T_n-\{x,y\}$ with $4$ colors. Also, when counting the $C_3^+$ in $T_n-y$ containing $x$, we have to account for the colors induced by the arcs with $y$.

\bigskip

	Objective: 
		maximize 
\[
\left(		
\begin{tikzpicture}[baseline=-0.5ex]
\node (1) at (0,-0.5) {};
\node (2) at (0,0.5) {};
\node (1') at (0,-0.5) {$\II$};
\node (2') at (0,0.5) {$\IO$};
\draw[dedge] (1) -> (2);
\end{tikzpicture}
+
\begin{tikzpicture}[baseline=-0.5ex]
\node (1) at (0,-0.5) {};
\node (2) at (0,0.5)  {};
\node (1') at (0,-0.5) {$\OO$};
\node (2') at (0,0.5)  {$\OI$};
\draw[dedge] (1) -> (2);
\end{tikzpicture}
\right)
-
\left(
\begin{tikzpicture}[baseline=-0.5ex]
\node (1) at (0,-0.5) {};
\node (2) at (0,0.5)  {};
\node (1') at (0,-0.5) {$\II$};
\node (2') at (0,0.5)  {$\OI$};
\draw[dedge] (1) -> (2);
\end{tikzpicture}
+
\begin{tikzpicture}[baseline=-0.5ex]
\node (1) at (0,-0.5) {};
\node (2) at (0,0.5) {};
\node (1') at (0,-0.5) {$\OO$};
\node (2') at (0,0.5)  {$\IO$};
\draw[dedge] (1) -> (2);
\end{tikzpicture}
\right)
.\]

Constraints: 
	\begin{align*}
		0.416 \leq&~ \OO+\OI \leq 0.44057 \\
		0.8849 \leq&~ \OO+\IO \\
		0.157500667 \leq&~ \mbox{ sum of all 4-colorings of }C_3^+ \\
		0.157500667 \leq&~ 
		\scalebox{0.5}{\begin{tikzpicture}[baseline=2.6ex]
				\node[shape=circle split, circle split part fill={black,black}, draw=black] (1) at (0,0) {}; 
				\node[shape=circle split, circle split part fill={black,black}, draw=black] (2) at (1,2) {}; 
				\node[shape=circle split, circle split part fill={black,black}, draw=black] (3) at (2,0) {}; 
				\draw[->,line width=2pt] (1) -> (2);
				\draw[->,line width=2pt] (2) -> (3);
				\draw[->,line width=2pt] (3) -> (1);
			\end{tikzpicture}}
			+
			\scalebox{0.5}{\begin{tikzpicture}[baseline=2.6ex]
				\node[shape=circle split, circle split part fill={black,black}, draw=black] (1) at (0,0) {}; 
				\node[shape=circle split, circle split part fill={black,black}, draw=black] (2) at (1,2) {}; 
				\node[shape=circle split, circle split part fill={black,white}, draw=black] (3) at (2,0) {}; 
				\draw[->,line width=2pt] (1) -> (2);
				\draw[->,line width=2pt] (2) -> (3);
				\draw[->,line width=2pt] (3) -> (1);
			\end{tikzpicture}}
			+
			\scalebox{0.5}{\begin{tikzpicture}[baseline=2.6ex]
				\node[shape=circle split, circle split part fill={black,black}, draw=black] (1) at (0,0) {}; 
				\node[shape=circle split, circle split part fill={black,white}, draw=black] (2) at (1,2) {}; 
				\node[shape=circle split, circle split part fill={black,white}, draw=black] (3) at (2,0) {}; 
				\draw[->,line width=2pt] (1) -> (2);
				\draw[->,line width=2pt] (2) -> (3);
				\draw[->,line width=2pt] (3) -> (1);
			\end{tikzpicture}}
			+
			\scalebox{0.5}{\begin{tikzpicture}[baseline=2.6ex]
				\node[shape=circle split, circle split part fill={black,white}, draw=black] (1) at (0,0) {}; 
				\node[shape=circle split, circle split part fill={black,white}, draw=black] (2) at (1,2) {}; 
				\node[shape=circle split, circle split part fill={black,white}, draw=black] (3) at (2,0) {}; 
				\draw[->,line width=2pt] (1) -> (2);
				\draw[->,line width=2pt] (2) -> (3);
				\draw[->,line width=2pt] (3) -> (1);
			\end{tikzpicture}}
			+
			\scalebox{0.5}{\begin{tikzpicture}[baseline=2.6ex]
				\node[shape=circle split, circle split part fill={white,black}, draw=black] (1) at (0,0) {}; 
				\node[shape=circle split, circle split part fill={white,black}, draw=black] (2) at (1,2) {}; 
				\node[shape=circle split, circle split part fill={black,black}, draw=black] (3) at (2,0) {}; 
				\draw[->,line width=2pt] (2) -> (1);
				\draw[->,line width=2pt] (2) -> (3);
				\draw[->,line width=2pt] (3) -> (1);
			\end{tikzpicture}}
			+
			\scalebox{0.5}{\begin{tikzpicture}[baseline=2.6ex]
				\node[shape=circle split, circle split part fill={white,black}, draw=black] (1) at (0,0) {}; 
				\node[shape=circle split, circle split part fill={white,black}, draw=black] (2) at (1,2) {}; 
				\node[shape=circle split, circle split part fill={black,white}, draw=black] (3) at (2,0) {}; 
				\draw[->,line width=2pt] (2) -> (1);
				\draw[->,line width=2pt] (2) -> (3);
				\draw[->,line width=2pt] (3) -> (1);
			\end{tikzpicture}}
			\\
			&+\scalebox{0.5}{\begin{tikzpicture}[baseline=2.6ex]
				\node[shape=circle split, circle split part fill={white,black}, draw=black] (1) at (0,0) {}; 
				\node[shape=circle split, circle split part fill={white,white}, draw=black] (2) at (1,2) {}; 
				\node[shape=circle split, circle split part fill={black,black}, draw=black] (3) at (2,0) {}; 
				\draw[->,line width=2pt] (2) -> (1);
				\draw[->,line width=2pt] (2) -> (3);
				\draw[->,line width=2pt] (3) -> (1);
			\end{tikzpicture}}
			+\scalebox{0.5}{\begin{tikzpicture}[baseline=2.6ex]
				\node[shape=circle split, circle split part fill={white,black}, draw=black] (1) at (0,0) {}; 
				\node[shape=circle split, circle split part fill={white,white}, draw=black] (2) at (1,2) {}; 
				\node[shape=circle split, circle split part fill={black,white}, draw=black] (3) at (2,0) {}; 
				\draw[->,line width=2pt] (2) -> (1);
				\draw[->,line width=2pt] (2) -> (3);
				\draw[->,line width=2pt] (3) -> (1);
			\end{tikzpicture}}
			+\scalebox{0.5}{\begin{tikzpicture}[baseline=2.6ex]
				\node[shape=circle split, circle split part fill={white,white}, draw=black] (1) at (0,0) {}; 
				\node[shape=circle split, circle split part fill={white,black}, draw=black] (2) at (1,2) {}; 
				\node[shape=circle split, circle split part fill={black,black}, draw=black] (3) at (2,0) {}; 
				\draw[->,line width=2pt] (2) -> (1);
				\draw[->,line width=2pt] (2) -> (3);
				\draw[->,line width=2pt] (3) -> (1);
			\end{tikzpicture}}
			+\scalebox{0.5}{\begin{tikzpicture}[baseline=2.6ex]
				\node[shape=circle split, circle split part fill={white,white}, draw=black] (1) at (0,0) {}; 
				\node[shape=circle split, circle split part fill={white,black}, draw=black] (2) at (1,2) {}; 
				\node[shape=circle split, circle split part fill={black,white}, draw=black] (3) at (2,0) {}; 
				\draw[->,line width=2pt] (2) -> (1);
				\draw[->,line width=2pt] (2) -> (3);
				\draw[->,line width=2pt] (3) -> (1);
			\end{tikzpicture}}
			+\scalebox{0.5}{\begin{tikzpicture}[baseline=2.6ex]
				\node[shape=circle split, circle split part fill={white,white}, draw=black] (1) at (0,0) {}; 
				\node[shape=circle split, circle split part fill={white,white}, draw=black] (2) at (1,2) {}; 
				\node[shape=circle split, circle split part fill={black,black}, draw=black] (3) at (2,0) {}; 
				\draw[->,line width=2pt] (2) -> (1);
				\draw[->,line width=2pt] (2) -> (3);
				\draw[->,line width=2pt] (3) -> (1);
			\end{tikzpicture}}
			+\scalebox{0.5}{\begin{tikzpicture}[baseline=2.6ex]
				\node[shape=circle split, circle split part fill={white,white}, draw=black] (1) at (0,0) {}; 
				\node[shape=circle split, circle split part fill={white,white}, draw=black] (2) at (1,2) {}; 
				\node[shape=circle split, circle split part fill={black,white}, draw=black] (3) at (2,0) {}; 
				\draw[->,line width=2pt] (2) -> (1);
				\draw[->,line width=2pt] (2) -> (3);
				\draw[->,line width=2pt] (3) -> (1);
			\end{tikzpicture}}\\
 		0.157500667 \leq&~
		\scalebox{0.5}{\begin{tikzpicture}[baseline=2.6ex]
				\node[shape=circle split, circle split part fill={black,black}, draw=black] (1) at (0,0) {}; 
				\node[shape=circle split, circle split part fill={black,black}, draw=black] (2) at (1,2) {}; 
				\node[shape=circle split, circle split part fill={black,black}, draw=black] (3) at (2,0) {}; 
				\draw[->,line width=2pt] (1) -> (2);
				\draw[->,line width=2pt] (2) -> (3);
				\draw[->,line width=2pt] (3) -> (1);
			\end{tikzpicture}}
			+
			\scalebox{0.5}{\begin{tikzpicture}[baseline=2.6ex]
				\node[shape=circle split, circle split part fill={black,black}, draw=black] (1) at (0,0) {}; 
				\node[shape=circle split, circle split part fill={black,black}, draw=black] (2) at (1,2) {}; 
				\node[shape=circle split, circle split part fill={white,black}, draw=black] (3) at (2,0) {}; 
				\draw[->,line width=2pt] (1) -> (2);
				\draw[->,line width=2pt] (2) -> (3);
				\draw[->,line width=2pt] (3) -> (1);
			\end{tikzpicture}}
			+
			\scalebox{0.5}{\begin{tikzpicture}[baseline=2.6ex]
				\node[shape=circle split, circle split part fill={black,black}, draw=black] (1) at (0,0) {}; 
				\node[shape=circle split, circle split part fill={white,black}, draw=black] (2) at (1,2) {}; 
				\node[shape=circle split, circle split part fill={white,black}, draw=black] (3) at (2,0) {}; 
				\draw[->,line width=2pt] (1) -> (2);
				\draw[->,line width=2pt] (2) -> (3);
				\draw[->,line width=2pt] (3) -> (1);
			\end{tikzpicture}}
			+
			\scalebox{0.5}{\begin{tikzpicture}[baseline=2.6ex]
				\node[shape=circle split, circle split part fill={white,black}, draw=black] (1) at (0,0) {}; 
				\node[shape=circle split, circle split part fill={white,black}, draw=black] (2) at (1,2) {}; 
				\node[shape=circle split, circle split part fill={white,black}, draw=black] (3) at (2,0) {}; 
				\draw[->,line width=2pt] (1) -> (2);
				\draw[->,line width=2pt] (2) -> (3);
				\draw[->,line width=2pt] (3) -> (1);
			\end{tikzpicture}}
			+
			\scalebox{0.5}{\begin{tikzpicture}[baseline=2.6ex]
				\node[shape=circle split, circle split part fill={black,white}, draw=black] (1) at (0,0) {}; 
				\node[shape=circle split, circle split part fill={black,white}, draw=black] (2) at (1,2) {}; 
				\node[shape=circle split, circle split part fill={black,black}, draw=black] (3) at (2,0) {}; 
				\draw[->,line width=2pt] (2) -> (1);
				\draw[->,line width=2pt] (2) -> (3);
				\draw[->,line width=2pt] (3) -> (1);
			\end{tikzpicture}}
			+
			\scalebox{0.5}{\begin{tikzpicture}[baseline=2.6ex]
				\node[shape=circle split, circle split part fill={black,white}, draw=black] (1) at (0,0) {}; 
				\node[shape=circle split, circle split part fill={black,white}, draw=black] (2) at (1,2) {}; 
				\node[shape=circle split, circle split part fill={white,black}, draw=black] (3) at (2,0) {}; 
				\draw[->,line width=2pt] (2) -> (1);
				\draw[->,line width=2pt] (2) -> (3);
				\draw[->,line width=2pt] (3) -> (1);
			\end{tikzpicture}}
			\\
			&+\scalebox{0.5}{\begin{tikzpicture}[baseline=2.6ex]
				\node[shape=circle split, circle split part fill={black,white}, draw=black] (1) at (0,0) {}; 
				\node[shape=circle split, circle split part fill={white,white}, draw=black] (2) at (1,2) {}; 
				\node[shape=circle split, circle split part fill={black,black}, draw=black] (3) at (2,0) {}; 
				\draw[->,line width=2pt] (2) -> (1);
				\draw[->,line width=2pt] (2) -> (3);
				\draw[->,line width=2pt] (3) -> (1);
			\end{tikzpicture}}
			+\scalebox{0.5}{\begin{tikzpicture}[baseline=2.6ex]
				\node[shape=circle split, circle split part fill={black,white}, draw=black] (1) at (0,0) {}; 
				\node[shape=circle split, circle split part fill={white,white}, draw=black] (2) at (1,2) {}; 
				\node[shape=circle split, circle split part fill={white,black}, draw=black] (3) at (2,0) {}; 
				\draw[->,line width=2pt] (2) -> (1);
				\draw[->,line width=2pt] (2) -> (3);
				\draw[->,line width=2pt] (3) -> (1);
			\end{tikzpicture}}
			+\scalebox{0.5}{\begin{tikzpicture}[baseline=2.6ex]
				\node[shape=circle split, circle split part fill={white,white}, draw=black] (1) at (0,0) {}; 
				\node[shape=circle split, circle split part fill={black,white}, draw=black] (2) at (1,2) {}; 
				\node[shape=circle split, circle split part fill={black,black}, draw=black] (3) at (2,0) {}; 
				\draw[->,line width=2pt] (2) -> (1);
				\draw[->,line width=2pt] (2) -> (3);
				\draw[->,line width=2pt] (3) -> (1);
			\end{tikzpicture}}
			+\scalebox{0.5}{\begin{tikzpicture}[baseline=2.6ex]
				\node[shape=circle split, circle split part fill={white,white}, draw=black] (1) at (0,0) {}; 
				\node[shape=circle split, circle split part fill={black,white}, draw=black] (2) at (1,2) {}; 
				\node[shape=circle split, circle split part fill={white,black}, draw=black] (3) at (2,0) {}; 
				\draw[->,line width=2pt] (2) -> (1);
				\draw[->,line width=2pt] (2) -> (3);
				\draw[->,line width=2pt] (3) -> (1);
			\end{tikzpicture}}
			+\scalebox{0.5}{\begin{tikzpicture}[baseline=2.6ex]
				\node[shape=circle split, circle split part fill={white,white}, draw=black] (1) at (0,0) {}; 
				\node[shape=circle split, circle split part fill={white,white}, draw=black] (2) at (1,2) {}; 
				\node[shape=circle split, circle split part fill={black,black}, draw=black] (3) at (2,0) {}; 
				\draw[->,line width=2pt] (2) -> (1);
				\draw[->,line width=2pt] (2) -> (3);
				\draw[->,line width=2pt] (3) -> (1);
			\end{tikzpicture}}
			+\scalebox{0.5}{\begin{tikzpicture}[baseline=2.6ex]
				\node[shape=circle split, circle split part fill={white,white}, draw=black] (1) at (0,0) {}; 
				\node[shape=circle split, circle split part fill={white,white}, draw=black] (2) at (1,2) {}; 
				\node[shape=circle split, circle split part fill={white,black}, draw=black] (3) at (2,0) {}; 
				\draw[->,line width=2pt] (2) -> (1);
				\draw[->,line width=2pt] (2) -> (3);
				\draw[->,line width=2pt] (3) -> (1);
			\end{tikzpicture}}
\end{align*}

\bigskip

We find that the solution to this program is bounded above by $-0.077$:
\[
\frac{-94611767053769387132533069422799745047344297}{1220703125000000000000000000000000000000000000}
 \approx -0.077506,
\]
implying that $(y \to x)$ results in at least $0.077 \cdot \binom{n}{2}$ more copies of $C_3^+$ than $(x \to y)$ in $T_n$ for sufficiently large $n$, proving our claim. Certificates can be found at \oururl.
\end{proof}

Having determined the behavior of the relationship between $H_n$ and $L_n$, we now focus on the internal behavior of $H_n$. The following claim implies that, for large enough $n$, the overall structure of $T_n$ iterates into $H_n$.

\begin{claim}\label{cl:Hn}
 $I(C_3^+,T_n[H_n]) = I(C_3^+,|H_n|)$.
 \end{claim}

\begin{proof}
The only copies of $C_3^+$ in $T_n$ are those chosen completely in $H_n$, completely in $L_n$, or with precisely 1 vertex chosen from $H_n$. The arcs in $T_n[H]$ impact neither the second nor third type of $C_3^+$. 
Therefore, $T_n[H_n]$ is extremal and the claim follows.
\end{proof}

We next focus on showing that the sizes of $H_n$ and $L_n$ are correct. While we could prove a slightly stronger bound here with the same method, we only need $|L_n|< (\frac67-\varepsilon)n$ later in Claim~\ref{QR2}.

\begin{claim}
$|L_n|< \frac67n-0.00001n\approx 0.85713n$.
\end{claim}

\noindent
\begin{proof}
First, since there are no arcs from $L_n$ to $H_n$, the average out-degree of vertices in $L_n$ is $\frac{|L_n|-1}{2}$, so by Claim~\ref{support}
\[
\tfrac1{n}|L_n| \in (0.832,0.88114).
\]
We would like a tighter upper bound, so we pose the following program wherein we color the vertices in $L_n$ black and the vertices in $H_n$ white. In this program, we assume that $|L_n|\ge \frac67n-\varepsilon$ and show that the density of $C_3^+$ is then bounded above by a bound smaller than in our construction, implying the claim. We note as well that $\begin{tikzpicture} \node[shape=circle,draw=black,fill=black,inner sep=2pt] (1) at (0,0){}; \node[shape=circle,draw=black,fill=white,inner sep=2pt] (2) at (0.7,0){}; \draw[dedge] (1) -> (2); \end{tikzpicture}$ is a forbidden subgraph by Claim \ref{cut}, so we include this as a constraint in the program as well.
\begin{quote} 
	Objective: \begin{quote}
		maximize \begin{tikzpicture}[baseline=1.3ex,scale=0.65]
			\node[shape=circle,draw=black,fill=black,inner sep=2pt] (1) at (0,0) {};
			\node[shape=circle,draw=black,fill=black,inner sep=2pt] (2) at (1,0) {};
			\node[shape=circle,draw=black,fill=black,inner sep=2pt] (3) at (1,1) {};
			\node[shape=circle,draw=black,fill=black,inner sep=2pt] (4) at (0,1) {};
			\draw[dedge] (1) -> (2);
			\draw[dedge] (2) -> (3);
			\draw[dedge] (3) -> (4);
			\draw[dedge] (4) -> (2);
			\draw[dedge] (1) -> (3);
			\draw[dedge] (1) -> (4);
		\end{tikzpicture} + \begin{tikzpicture}[baseline=1.3ex,scale=0.65]
			\node[shape=circle,draw=black,fill=white,inner sep=2pt] (1) at (0,0) {};
			\node[shape=circle,draw=black,fill=black,inner sep=2pt] (2) at (1,0) {};
			\node[shape=circle,draw=black,fill=black,inner sep=2pt] (3) at (1,1) {};
			\node[shape=circle,draw=black,fill=black,inner sep=2pt] (4) at (0,1) {};
			\draw[dedge] (1) -> (2);
			\draw[dedge] (2) -> (3);
			\draw[dedge] (3) -> (4);
			\draw[dedge] (4) -> (2);
			\draw[dedge] (1) -> (3);
			\draw[dedge] (1) -> (4);
		\end{tikzpicture} + \begin{tikzpicture}[baseline=1.3ex,scale=0.65]
			\node[shape=circle,draw=black,fill=white,inner sep=2pt] (1) at (0,0) {};
			\node[shape=circle,draw=black,fill=white,inner sep=2pt] (2) at (1,0) {};
			\node[shape=circle,draw=black,fill=white,inner sep=2pt] (3) at (1,1) {};
			\node[shape=circle,draw=black,fill=white,inner sep=2pt] (4) at (0,1) {};
			\draw[dedge] (1) -> (2);
			\draw[dedge] (2) -> (3);
			\draw[dedge] (3) -> (4);
			\draw[dedge] (4) -> (2);
			\draw[dedge] (1) -> (3);
			\draw[dedge] (1) -> (4);
		\end{tikzpicture}
		\end{quote}
	Constraints: \begin{quote}
		$6/7-0.00001 \leq \begin{tikzpicture}
			\node[shape=circle,draw=black,fill=black,inner sep=2pt] (1) at (0,0) {};
		\end{tikzpicture}\le 0.88114$ 
		\end{quote}
	Forbidden Subgraph: \begin{quote}
		\begin{tikzpicture}
			\node[shape=circle,draw=black,fill=black,inner sep=2pt] (1) at (0,0) {};
			\node[shape=circle,draw=black,fill=white,inner sep=2pt] (2) at (0.7,0) {};
			\draw[dedge] (1) -> (2);
		\end{tikzpicture}
		\end{quote}
\end{quote}
This program is bounded above by 
\[
\frac{141750035467643236211349281693358548643157719744000}{900000000000000000000000000000000000000000000000000 }  < 0.15750004<i(C_3^+,T_n),
\]
thus cannot be extremal, and so implies that $T_n$ must satisfy $|L_n| < \frac67n-0.00001$. Certificates can be found at \oururl.
\end{proof}

We next aim to prove that the sequence $(T_n[L_n])_{n=1}^\infty$ is quasi-random. To do so we prove Claim~\ref{QR}, a consequence of the characterization of quasi-random tournaments in (Chung and Graham~\cite{MR1106530}). They list 11 different equivalent properties characterizing quasi-random sequences $G_n$ of tournaments on $n$ vertices, but we will only use the first two.
\begin{itemize}
\item[$P_1$:]  For every fixed tournament $H$, $i(H,G_n)$ converges to the expected value in the random tournament on $n$ vertices.
\item[$P_2$:] $\displaystyle{\lim_{n\to\infty}i(C_4,G_n) = \tfrac{3}{8}}$.
\end{itemize}
Since graphons are completely determined by subgraph densities, $P_1$ implies that every quasi-random sequence of tournaments has the same limit graphon. The following claim is an addition to the 11 properties listed in~\cite{MR1106530}, tailored to our problem.

\begin{claim}\label{QR}
A sequence of tournaments $(G_n)_{n=1}^\infty$ with $|G| = n$ is quasi-random if and only if 
\[
\lim_{n\to\infty} i(C_3,G_n) = \tfrac{1}{4}~ \text{ and }
	\lim_{n\to\infty} i(C_3^+,G_n)= \tfrac{1}{8}.
\]
\end{claim}

\noindent
\begin{proof} The "only if" statement follows immediately from property $P_1$, so we  concern ourselves with proving the "if" statement. Let $(G_n)_{n=1}^\infty$ be a sequence of tournaments so that $|G_n| = n$, and recall from Proposition~\ref{C3} that (near) regular tournaments are the tournaments which maximize the number of induced copies of $C_3$.

So, assume that 
\[
\lim_{n\to\infty} i(C_3,G_n) = \tfrac{1}{4}~ \text{ and }
	\lim_{n\to\infty} i(C_3^+,G_n)= \tfrac{1}{8},
\]
 and observe that this implies that the degrees in the tournaments are concentrated around $\frac{n}2$, i.e. all but $o(n)$ vertices have out-degree $\frac{n}2+o(n)$. Now observe 
 that 
 \begin{align*}
\frac{1}{8}\binom{n}{4} + o(n^4)&= I(C_3^+,G_n) \\
&= \sum_{v \in V(G_n)} I(C_3,G_n[N^+(v)]) \\
	&= \sum_{v \in V(G_n)} i(C_3,G_n[N^+(v)])\binom{n/2}{3} + o(n^3),\text{ by degree concentration} \\
	&\leq \sum_{v \in V(G_n)} \left(\frac{1}{4}+o(1)\right)\binom{n/2}{3} + o(n^3), \text{ by the inducibilty of $C_3$} \\
	&= \frac{1}{8}\binom{n}{4} + o(n^4).
\end{align*}
This implies that $i(C_3,G_n[N^+(v)])=\frac14+o(1)$  for all but at most $o(n)$ vertices $v \in V(G_n)$. This equality also implies that 
$ i(TT_4,G_n)= \frac{3}{8}+o(1)$. Now
\begin{align*}
\frac14+o(1)&=i(C_3,G_n)=\frac12 i(C_4,G_n)+\frac14 i(C_3^+,G_n)+\frac14 i(C_3^-,G_n), \mbox{ and}\\
\frac14+o(1)&=\frac13 i(TT_3,G_n)=\frac16 i(C_4,G_n)+\frac14 i(C_3^+,G_n)+\frac14 i(C_3^-,G_n)+\frac13 i(TT_4,G_n),
\end{align*}
so
\[
o(1)=i(C_3,G_n)-\frac13 i(TT_3,G_n)=\frac13i(C_4,G_n)-\frac18+o(1),
\]
and thus $i(C_4,G_n)=\frac38+o(1)$. This last statement is equivalent to property $P_2$.
%
%
\end{proof}

\begin{claim}\label{mainIE}
In any tournament $T$ on $n$ vertices, $i(C_3^+,T)\le  \frac{1}{8} + \frac{2}{3}(\frac{1}{4} - i(C_3,T))+o(1)$.
\end{claim}
\begin{proof}
Using the plain flag algebra method, we show
\[
3i(C_3^+,T) +  2i(C_3,T) \leq \frac{7}{8}+o(1).
\]
The claim follows by rearranging the inequality.
Certificates can be found at \oururl.
\end{proof}

\begin{claim}\label{QR2}
The sequence $(T_n[L_n])$ is quasi-random.
\end{claim}

\noindent
\begin{proof} 
%
Let $L=\frac1n |L_n|$. 
By Claim~\ref{cl:Hn}, $i(C_3^+,T_n[H_n])=i(C_3^+,T_n)+o(1)$. Thus, the density of the $C_3^+$ which are not completely contained in $H_n$ is $(1-(1-L)^4)i(C_3^+,T_n)+o(1)$. We have, 
\begin{align*}
L^4\tfrac18+4(1-L)L^3\tfrac14 &\le
(1-(1-L)^4)i(C_3^+,T_n) +o(1)\\
&= L^4\cdot i(C_3^+,T_n[L_n]) + 4(1-L)L^3 \cdot i(C_3,T_n[L_n]) +o(1)\\
	&\leq L^4\cdot \op{\tfrac{1}{8} + \tfrac{2}{3}(\tfrac14 - i(C_3,T_n[L_n]))} + 4(1-L)L^3 \cdot i(C_3,T_n[L_n]) +o(1)\\
	&= \tfrac{7}{24} L^4+ L^3i(C_3,T_n[L_n])(4 - \tfrac{14}{3}L)+o(1)\\
	&\le \tfrac{7}{24} L^4+ L^3\tfrac14 (4 - \tfrac{14}{3}L) +o(1)\\
	&= L^4\tfrac18+4(1-L)L^3\tfrac14 +o(1).
\end{align*}
The first inequality is true as the left side is the value of the next term we would expect if we replaced $T_n[L_n]$ by a random tournament on the same vertices. The second inequality follows from Claim~\ref{mainIE}. For the last inequality, note that $0<L<\frac67-0.00001+o(1)$ implying $4 - \tfrac{14}{3}L>0.00004+o(1)$. Thus, the left side is maximized if and only if $C_3$ is maximized at $\frac14$. As the first and the last term in this chain of inequalities are equal up to $o(1)$, we have equality throughout. Thus  
$i(C_3,T_n[L_n])=\frac14+o(1)$ and $i(C_3^+,T_n[L_n])=\frac18+o(1)$, proving the claim using Claim~\ref{QR}.
\end{proof}

\begin{claim}
The normalized size of $L_n$ is $L = \frac{1}{5}\op{7 + \sqrt[3]{3} - 2\sqrt[3]{9}} + o(1)$, and our construction converges in the graphon language to the limit object for the inducibility of $C_3^+$.
\end{claim}

\noindent
\begin{proof}
We know that $L < 6/7+o(1)$, that $T_n[L_n]$ is quasi-random, that all arcs between $H_n$ and $L_n$ point towards $L_n$, and that $i(C_3^+,T_n[H_n])=i(C_3^+,T_n)+o(1)$ since $T_n[H_n]$ is extremal for $C_3^+$. Thus, 
\[
i(C_3^+,T_n) = L^4 \cdot \frac{1}{8} + \binom{4}{1}L^3(1-L) \cdot \frac{1}{4} + (1-L)^4 (i(C_3,T_n)+o(1)).
\]
This is maximized when $i(C_3^+,T_n)  = \frac{1}{8}\op{8 - 9\sqrt[3]{3} + \sqrt[3]{3^5}}+o(1)$ and $L= 
1-\alpha+o(1)$.
\end{proof}

We have thus shown that every extremal tournament matches our construction up to the choice of the sequence of quasi-random tournaments, completing the proof of this theorem.
\end{proof}

\section{Discussion}

In this section, we discuss some of the peculiarities of this problem and its solutions, including the novel strategies introduced in this paper. First and foremost, we know of no other inducibility problem for which all extremal constructions include a quasi-random component as in the case of $C_3^+$ and $C_3^-$ and ask the following question:

\begin{problem} For what classes of graphs (undirected or directed) do the extremal constructions for the corresponding inducibility problem involve non-trivial quasi-random components?
\end{problem}

For $C_3^+$, the extremal construction was conjectured by noting that our tournament can be decomposed into a source vertex and a $C_3$; described another way, we begin with an arc and blow up the head into a $C_3$. Essentially, we ask the following: for a digraph $G=(V,E)$ with cut $C = (S,T)$ and cut-set of size $|S| \cdot |T|$, for what structures $G[S]$ and $G[T]$ does the resulting inducibility problem have as extremal solutions constructions for which $\alpha \cdot 100\%$ of the vertices induce a ``typical random graph structure'' for some $\alpha \in (0,1)$? Natural candidates for consideration would include $G[T] \cong C_3$ and $G[S]$ isomorphic to any 2-vertex digraph or 3-vertex tournament.

Historically, flag algebra techniques have been leveraged to determine bounds on global graph densities. The models developed in Claims~\ref{support} and \ref{cut}, however, resulted in bounds on localized information. In the case of Claim~\ref{support}, we were able to determine something very powerful regarding the distribution of out-degrees in extremal constructions, namely that all vertices have normalized out-degrees in a very specific set. In the case of Claim~\ref{cut}, we were able to determine the direction of an arc between any pair of vertices which satisfy basic constraints related to their out-degrees.

Finally, we want to make an observation about Conjecture~\ref{conTC}. Let $k\ge 5$ be odd, and let $n>k$. Let $X\subset V(C_n)$ be a set of $k$ vertices such that $C_n[X]\cong C_k$. Observe that for every vertex $v\in V(C_n)\setminus X$, we have $C_n[X\cup \{x\}]\cong C_{k+1}$. If we now express $i(C_k,T)$ in a tournament $T$ in terms of densities of $(k+1)$-vertex graphs similarly to~\eqref{eqTC}, we can easily conclude that Conjecture~\ref{conTC} is true for $k+1$ if it is true for $k$, so it suffices to prove it for all odd $k$. Standard plain flag algebra computations give sharp bounds for $i(C_5)$ and $i(C_7)$, and further show that $C_n$ is $o(n^2)$ arc flips away from every extremal tournament for $C_5$ and $C_7$ (and thus for $C_6$ and $C_8$ by this observation), but we have not seriously tried to show the full conjecture for these cases, which would require to exactly determine the extremal tournaments. 

\section*{Acknowledgment}
This work used the computing resources at the Center for Computational Mathematics, University of Colorado Denver, including the Alderaan cluster, supported by the National Science Foundation award OAC-2019089.

\bibliographystyle{plainurl}
\bibliography{references}

\end{document}